\documentclass[reqno]{amsart}
\usepackage[utf8]{inputenc}
\usepackage{amssymb, amsthm}
\usepackage{mathrsfs}
\usepackage{bm,bbm}
\usepackage{MnSymbol}
\usepackage{microtype} 
\usepackage{comment} 

\usepackage[left= 3.4 cm, right= 3.4 cm, top= 2.8 cm, bottom=2.7 cm, foot= 1.2 cm, marginparwidth=2.7 cm, marginparsep=0.5 cm]{geometry}

\usepackage{amsfonts}
\DeclareMathAlphabet{\mathpgoth}{OT1}{pgoth}{m}{n}
\DeclareMathAlphabet{\mathesstixfrak}{U}{esstixfrak}{m}{n}
\DeclareMathAlphabet{\mathboondoxfrak}{U}{BOONDOX-frak}{m}{n}
\usepackage{amsmath,amssymb}
\usepackage[bbgreekl]{mathbbol}
\usepackage{yfonts,mathtools}

\usepackage{tikz-cd}
\usepackage[all,cmtip]{xy}

\numberwithin{equation}{section}
\usepackage{soul}

\usepackage{hyperref}
\hypersetup{colorlinks}
\definecolor{darkred}{rgb}{0.5,0,0}
\definecolor{darkgreen}{rgb}{0,0.5,0}
\definecolor{darkblue}{rgb}{0,0,0.5}
\hypersetup{colorlinks, linkcolor=darkblue, filecolor=darkgreen, urlcolor=darkred, citecolor=darkblue}
\makeatletter 
\@addtoreset{equation}{section}
\makeatother  

\numberwithin{equation}{section}

\newtheorem{thma}{Theorem}

\newtheorem{thm}{Theorem}[section]
\newtheorem{cor}[thm]{Corollary}

\newtheorem{prop}[thm]{Proposition}

\newtheorem{lemma}[thm]{Lemma}
\theoremstyle{definition}
\newtheorem{defn}[thm]{Definition}
\theoremstyle{remark}
\newtheorem{rem}[thm]{Remark}

\usepackage{xcolor}

\newcommand{\beq}{\begin{equation}}
\newcommand{\eeq}{\end{equation}}
\newcommand{\beqn}{\begin{equation*}}
\newcommand{\eeqn}{\end{equation*}}
\newcommand{\ov}{\overline}
\newcommand{\mb}{\mathbb}
\newcommand{\mc}{\mathcal}
\newcommand{\mf}{\mathfrak}

\newcommand{\GW}{{\rm GW}}
\newcommand{\GGW}{\widetilde{\rm GW}{}}
\newcommand{\pt}{{\rm pt}}

\newcommand{\wt}{\widetilde}

\newcommand{\wh}{\widehat}

\newcommand{\uds}[1]{\underline{\smash{#1}}}

\newcommand{\ev}{{\rm ev}}

\title{Cohomological splitting over rationally connected bases}

\author{Shaoyun Bai}
\address{MIT, 77 Massachusetts Avenue Cambridge, MA 02139, USA}
\email{shaoyunb@mit.edu}

\author{Daniel Pomerleano}
\address{University of Massachusetts, Boston, 100 William T, Morrissey Blvd, Boston, MA 02125, USA}
\email{Daniel.Pomerleano@umb.edu}

\author{Guangbo Xu}
\address{Department of Mathematics, Rutgers University, Hill Center--Busch Campus, 110 Frelinghuysen Road, Piscataway, NJ 08854-8019, USA}
\email{guangbo.xu@rutgers.edu}

\thanks{D.P. is supported by NSF DMS-2306204.}

\thanks{G.X. is supported by NSF DMS-2345030.}

\begin{document}

\begin{abstract}
We prove a cohomological splitting result for Hamiltonian fibrations over enumeratively rationally connected symplectic manifolds.   As a key application, we prove that the cohomology of a smooth, projective family over a smooth (stably) rational projective variety splits additively over any field. The main ingredients in our arguments include the theory of Fukaya--Ono--Parker (FOP) perturbations developed by the first and third author, which allows one to define integer-valued Gromov--Witten type invariants, and variants of Abouzaid--McLean--Smith's global Kuranishi charts tailored to concrete geometric problems.
\end{abstract}


\maketitle

\setcounter{tocdepth}{1}
\tableofcontents

\section{Introduction}


This paper is concerned with the classical Leray spectral sequence in the context of symplectic and birational geometry. Recall that for any fibration $\pi_P: P \to B$ of topological spaces with fiber $M$, for any coefficient field ${\mb k}$, one has a Leray spectral sequence 
\beq
E_2^{p,q}:= H^p(B; {\bf R}^q \pi_P({\mb k})) \Longrightarrow H^{p+q}(P; {\mb k}).
\eeq



A recurring theme in the literature is that the rigidity of symplectic or complex geometry often forces the spectral sequence to degenerate at the $E_2$-page. This is the case, for example, when $\pi_P$ is a smooth projective morphism between smooth complex projective varieties and when $\mathbb{k}$ is of characteristic zero, according to a celebrated result of Blanchard and Deligne \cite{Blanchard, Deligne}. In symplectic topology, Lalonde, McDuff, Polterovich \cite{LMP, mcduffseidel, Lalonde_McDuff_2003} have shown that degeneration in characteristic zero also holds for certain Hamiltonian fibrations, though it is a more flexible category than projective fibrations. In fact, for the Hamiltonian fibrations they consider, Lalonde, McDuff, Polterovich prove something \emph{a priori} stronger than degeneration; they prove that there is a cohomological splitting: 
\begin{align} H^*(B; \mathbb{Q})\otimes H^*(M; \mathbb{Q}) \cong H^{p+q}(P; \mathbb{Q}). \end{align}

Recently, Abouzaid--McLean--Smith \cite{AMS} proved a cohomological splitting result for Hamiltonian fibrations over $S^2$ in all characteristics as applications of remarkable advances in symplectic geometry.\footnote{See also the work of \cite{Bai_Xu_2022} for an alternate proof of the theorem of Abouzaid--McLean--Smith.} The purpose of this article is to exhibit a broader class of situations in both the symplectic and algebraic categories when this cohomological splitting occurs for coefficient fields with positive characteristic. The main condition we need, roughly speaking, is that the base $B$ is rationally connected, a concept from algebraic geometry. Rationally connected varieties are often considered to be the ``correct higher dimensional  analogs of rational curves"(c.f. \cite[p. 425]{kollar2001}). The main tools that we use are from symplectic enumerative geometry: moduli spaces of stable maps and Gromov--Witten invariants.

\subsection{Statement of results}

We now give a precise definition of the rational connectedness hypothesis we need in our first result. Recall that an algebraic variety is called rationally connected if every two points in this variety are connected by a rational curve. For example, all smooth Fano varieties 
are rationally connected by a famous theorem of Koll\'ar--Miyaoka--Mori \cite{Kollar_Miyaoka_Mori_1992} and Campana \cite{campana}. A conjecture of Koll\'ar \cite[Conjecture 4.2.7]{kollar-ems} 
which says that rational connectedness is symplectic deformation invariant, as well as work on this conjecture by 
Voisin \cite{voisin-threefold} and Z. Tian \cite{tian-threefold, tian-fourfold}, motivates the following symplectic definition of rational connectedness.

\begin{defn}\label{defn:rational-connected}
Let $(B, \omega_B)$ be a closed symplectic manifold. We say that $B$ is \emph{symplectically or enumeratively rationally connected} if there exists $k \in {\mb Z}_{\geq 0}$ such that it admits a genus $0$ Gromov--Witten invariant
    \begin{equation}\label{eqn:GW-ptpt}
        \mathrm{GW}_{0,k+2}^{B, A}([ \pt], [ \pt], a_1, \cdots, a_k) \neq 0,
    \end{equation}
    where $[\pt] \in H^{\rm top}(B ; {\mb Q})$ is the Poincar\'e dual of the point class, $a_i \in H^*(B; {\mb Q})$, and $A \in H_2(B; {\mb Z})$ is a curve class.
\end{defn}


To state our first result, we focus on monotone symplectic manifolds, i.e., those whose first Chern class is a positive multiple of the symplectic class. These manifolds 
are the symplectic counterparts of Fano varieties. In this case, we can assume in Definition \ref{defn:rational-connected} that the GW invariant is an integer with $a_i \in H^*(B; {\mb Z})$ (cf. \cite{Ruan_Tian, McDuff_Salamon_2004}).

\begin{thma}\label{thm:split}
Let $(B, \omega_B)$ be a compact, monotone, and enumeratively rationally connected symplectic manifold with a nonzero Gromov--Witten invariant $\mathrm{GW}_{0,k+2}^{B, A}([\pt], [\pt], a_1, \cdots, a_k)$. If $(M, \omega_M)$ is a compact symplectic manifold and 
\beqn
\xymatrix{ M \ar[r] & P \ar[d]^{\pi_P} \\
            &         B }
\eeqn
is a Hamiltonian fibration, then for any coefficient field $\mathbb{k}$ whose characteristic 
does not divide $\mathrm{GW}_{0,k+2}^{B, A}([\pt], [\pt], a_1, \cdots, a_k) \in {\mb Z}$, there is an isomorphism of graded ${\mb k}$-vector spaces:   \begin{equation}\label{eqn:coh-splitting}
        H^*(P; \mathbb{k}) \cong H^*(B; \mathbb{k}) \otimes_{\mathbb{k}} H^*(M; \mathbb{k}).
    \end{equation}
\end{thma}

Returning to our original algebro-geometric situation, Theorem \ref{thm:split} implies the following result. 

\begin{cor}\label{cor12}
Let $\pi_P: P \to B$ be a smooth, projective morphism from a smooth complex projective variety $P$ to a smooth complex projective base $B$ which is Fano and enumeratively rationally connected. Then for any coefficient field ${\mb k}$, $H^*(P; {\mb k}) \cong H^*(B; {\mb k}) \otimes_{\mb k} H^*(M; {\mb k})$ provided that there is a genus zero primary Gromov--Witten invariant of $B$ with at least two point insertions which is not divisible by ${\rm char}({\mb k})$. 
\end{cor}

\begin{proof}
By the theorem of Koll\'ar--Miyaoka--Mori \cite{Kollar_Miyaoka_Mori_1992} and Campana \cite{campana}, $\pi_1(B) = 0$. On the other hand, as $P$ is projective, there is a K\"ahler form $\omega_P \in \Omega^2(P)$; as $\pi_P$ is a holomorphic submersion, each fiber is a complex, hence symplectic submanifold. Then by \cite[Lemma 6.2]{McDuff_Salamon_1998} (or see \cite[Theorem 1.7]{BGS2} for K\"ahler fibrations), the map $\pi_P: P \to B$ admits the structure of symplectic fibration. Then by \cite[Theorem 1.1]{Lalonde_McDuff_2003}, such a symplectic fibration reduces to a Hamiltonian fibration. Therefore, Theorem \ref{thm:split} implies the splitting result. 
\end{proof}

Note that we do not impose any positivity condition on the fiber $(M, \omega_M)$. At this level of generality, Theorem \ref{thm:split} was only known for $B = \mathbb{CP}^1$ by Abouzaid--McLean--Smith \cite[Theorem 1.1]{AMS} (see the work of the first and third author \cite{Bai_Xu_2022} for an alternative proof), and $B$ being a product of Grassmannians by a result of the first and second author \cite[Lemma 2.8]{Bai_Pomerleano_2024}. To compare these previous results with Theorem \ref{thm:split}, note that for $\mathbb{CP}^1$, the Gromov--Witten invariant $\mathrm{GW}_{0,2}^{\mathbb{CP}^1, [\mathbb{CP}^1]}([\pt], [\pt]) = 1$ because itself is the unique line connecting two points; for the Grassmannians, as discussed in \cite[Example 3.10]{hu-rational-connected}, there is a nonvanishing Gromov--Witten invariant $\mathrm{GW}_{0,3}^{B, A}([\pt], [\pt], a) = 1$, so $\eqref{eqn:coh-splitting}$ holds for any field $\mathbb{k}$ when $B$ is either $\mathbb{CP}^1$ or a Grassmannian directly from Theorem \ref{thm:split}. 

\begin{rem} In Section \ref{section6} below, we provide an example of a projective fibration over a smooth Fano variety whose cohomology does not split in characteristic $2$. We expect, but do not prove, that this variety admits a genus zero primary Gromov--Witten invariant with two point insertions which is equal to $2$.\end{rem}

In proving Theorem \ref{thm:split}, the monotone condition of the base can be relaxed as long as the nonzero GW invariant is realized by a transversely cut out moduli space (see Theorem \ref{thm41}). This seemingly technical extension has a noteworthy consequence. Recall that a complex projective variety $B$ of dimension $m$ is {\it stably rational} if $B \times \mb{CP}^r$ is birational to $\mb{CP}^{m+r}$. It is not difficult to see that stably rational varieties are rationally connected \cite{Harris2001}. By combining the above methods with certain topological arguments from \cite{Lalonde_McDuff_2003,Bai_Pomerleano_2024}, we are able to show:  
%


\begin{thma}\label{thm:B}
Suppose $B$ is a smooth stably rational projective variety. Then for any Hamiltonian fibration $P \to B$ (in particular, any smooth algebraic family $P \to B$ with $P$ projective) and any coefficient field ${\mb k}$, there is an isomorphism of graded ${\mb k}$-vector spaces 
\beqn
H^*(P; {\mb k}) \cong H^*( B; {\mb k}) \otimes_{\mb k} H^*(M; {\mb k}).
\eeqn
\end{thma}


\begin{rem} 
Historically, the study of cohomological splitting of Hamiltonian fibrations over $\mathbb{CP}^1$ was initiated by Lalonde--McDuff--Polterovich \cite{LMP} and McDuff \cite{mcduffseidel} using holomorphic curve methods, building on Seidel's work on Seidel representations \cite{seidel-rep}. For general symplectic fibers, Lalonde--McDuff \cite{Lalonde_McDuff_2003} used various ``soft" topological arguments to establish the cohomological splitting over ${\mb Q}$ for Hamiltonian fibrations over a large class of manifolds, and proposed the relation between cohomological splitting and enumeratively rational connectedness \cite[Section 4.2]{Lalonde_McDuff_2003}. They established cohomological splitting over ${\mb Q}$ for any Hamiltonian fibration over blowups of $\mathbb{CP}^n$. The Lalonde--McDuff proposal was partially realized by Hyvrier \cite{hyvrier}, who proved the cohomological splitting with ${\mb Q}$ coefficients for Hamiltonian fibrations over symplectically rationally connected bases under certain restrictive assumptions. Theorem \ref{thm:split} and Theorem \ref{thm:B} are the first general results along this line of inquiry and also incorporate torsion classes in cohomology.
\end{rem}

\subsection{Discussion of the proof}

Because we work over a field, by the Leray--Hirsch theorem, we know that our result about the splitting of cohomology follows from 
the surjectivity of the restriction map
\beqn
H^*(P; \mathbb{k}) \to H^* (M; \mathbb{k}),
\eeqn
or equivalently the injectivity of the inclusion of homology. We prove this injectivity via a certain non-vanishing result of ${\mb k}$-valued Gromov--Witten invariants of the total space. The Hamiltonian fibration structure induces a closed extension of the fiberwise symplectic form to the total space $P$, which becomes a symplectic form after adding a multiple of the pullback of the symplectic form on the base $B$. 
When ${\rm char}({\mb k}) > 0$, the ${\mb k}$-valued Gromov--Witten invariants can be defined via the integral Gromov--Witten pseudocycle defined in \cite{Bai_Xu_2022} following the idea of Fukaya--Ono \cite{Fukaya_Ono_integer}. Based on this, Theorem \ref{thm:split} follows from the following nonvanishing result. To illustrate the idea of the proof without getting into the details we assume that $k = 0$ so the base $B$ admits a nonvanishing integral Gromov--Witten invariant $\GW_{0,2}^{B, A}({\rm pt}, {\rm pt}) \neq 0$.

\begin{thm}\label{thm14} (See a more general version as Theorem \ref{thm_nonvanishing})     Under the same assumptions as in Theorem \ref{thm:split} while assuming $k = 0$, for any nonzero homology class $\beta_0 \in H_{*}(M; {\mb k})$, there exists a homology class $\beta_\infty \in H_{*}(M; {\mb k})$ and a curve class $\tilde A \in H_2(P; {\mb Z})$ such that
    \begin{equation}
        \mathrm{GW}^{P, {\mb k}; \tilde A}_{0,2} \Big( {\rm PD}(\iota_*(\beta_0)), {\rm PD}(\iota_*(\beta_\infty)) \Big) \neq 0,
    \end{equation}
    where the above Gromov--Witten invariant is the ${\mb k}$-valued version induced from the integral Gromov--Witten invariants.
\end{thm}

We provide a heuristic explanation of this result and indicate the technical difficulties. Very ideally, the moduli space of two-pointed stable genus $0$ maps to $B$ representing the second homology class $A \in H_2(B;{\mb Z})$, denoted by $
\ov{\mc M}_{0,2}(B,A)$, is smooth and compact, with the evaluation map
\beqn
\ev: \ov{\mc M}_{0,2}(B,A) \to B \times B
\eeqn
a local diffeomorphism of finite degrees. In this case, for any pair of distinct points $p_1, p_2 \in B$, the Gromov--Witten invariant is a signed sum of the finitely many rational curves in $\ov{\mc M}_{0,2}(B,A)$ which pass through $p_1$ and $p_2$. If $\tilde{A} \in H_2(P;{\mb Z})$ is a lift of $A$, then we see that the composition with $\pi_P$ defines a map between the moduli spaces of stable rational curves
\beqn
\ov{\mc M}_{0,2}(P,\tilde{A}) \to \ov{\mc M}_{0,2}(B,A).
\eeqn
The preimage of an element $\ev^{-1}(\{p_1\} \times \{ p_2 \}) \in \ov{\mc M}_{0,2}(B,A)$, if can be represented by a smooth rational curve $C$ in $B$, can be identified with the moduli space of stable genus $0$ holomorphic ``sections" of the fibration over the projective line, $P|_{C}$. Then, in this ideal situation, we can appeal to the nonvanishing result of Gromov--Witten section invariants of fibrations over $\mathbb{CP}^1$ \cite{seidel-rep, mcduffseidel} to find the desired invariant, at least for the ${\mb Q}$-valued version.

The technical difficulties of carrying out the above sketch in general are of a familiar kind: to compactify the moduli space of rational curves using stable maps, one needs to introduce maps with nodal domains; also, the moduli spaces are ``stacky" in nature, which means that the nontrivial automorphisms of stable maps forbid one to go beyond characteristic zero when defining enumerative invariants following the traditional approach to Gromov--Witten theory.

We resolve these obstacles using tools from symplectic topology, both classical and modern. In the monotone setting as Theorem \ref{thm:split}, one uses the classical regularization method written up systematically by McDuff--Salamon \cite{McDuff_Salamon_2004} to find a generic almost complex structure on $B$ compatible with $\omega_B$ to ensure that only smooth curves in $\ov{\mc M}{}_{0,2}(B, A)$ contributes to the relevant Gromov--Witten invariant. In the situation of Theorem \ref{thm:B} where $B$ is not necessarily monotone, we reduce the consideration to the case that $B$ is a blowup of $\mb{CP}^n$, which has a nonzero Gromov--Witten invariant realized by a single transverse smooth rational curve. In order to allow the field ${\mb k}$ to have positive characteristic, we appeal to the Fukaya--Ono--Parker (FOP) perturbation method developed by the first and third author \cite{Bai_Xu_2022, Bai_Xu_Arnold} to develop a theory of Seidel elements in finite characteristic and find the required nonzero Gromov--Witten type invariant. These two regularization methods are combined fruitfully via variants of the global Kuranishi charts of Abouzaid--McLean--Smith \cite{AMS, AMS2}.

Finally,  we note that the approach 
of \cite{AMS} goes via Morava $K$-theories. It is an interesting open problem to see whether one can prove Theorem \ref{thm:split} using Gromov--Witten Morava $K$-theory invariants. We expect that to make progress along these lines, one needs to revisit the Baas--Sullivan approach \cite{Baas} to generalized cohomology theories to establish a dictionary between ordinary Gromov--Witten invariants and these generalized Gromov--Witten invariants. There are other motivations to investigate this question: via chromatic lifting techniques, one may go beyond additive splitting of ordinary cohomology. We note that \cite[Lemma 2.8]{Bai_Pomerleano_2024} implies that for any smooth projective family $P \to B$ over a product of Grassmannians and any complex oriented cohomology theory $\mathbb{E}^*$, we have: \begin{align} \mathbb{E}^*(P) \cong \mathbb{E}^*(M)\otimes_{\mathbb{E}^*(pt)}\mathbb{E}^*(B).\end{align} It would be interesting to know whether a similar splitting held in the stably rational case.

\subsection{Outline of the paper}

In Section \ref{section2}, we review the definition of integer-valued Gromov--Witten invariants given in \cite{Bai_Xu_2022} via the Fukaya--Ono--Parker (FOP) perturbation method, as well as the reduction to fields of finite characteristic. In Section \ref{section3}, we define the ${\mb k}$-valued graph Gromov--Witten invariants for Hamiltonian fibrations over $S^2$ and the corresponding Seidel map using the FOP perturbation method. In Section \ref{section4}, we state the main technical result (Theorem \ref{thm41}), which implies both Theorem \ref{thm:split} and Theorem \ref{thm:B}. In Section \ref{section5}, we prove Theorem \ref{thm41}. In Section \ref{section6}, we provide an example of a projective fibration over a Fano manifold whose cohomology does not split in characteristic $2$. 
\subsection*{Acknowledgement}
We would like to thank Sasha Kuznetsov for some help with Theorem 6.1, J{\o}rgen Vold Rennemo for a helpful email exchange concerning Remark \ref{rem:rennemo}, and Ziquan Zhuang for a discussion on weak factorizations. D.P. would like to thank Paul Seidel for an inspiring discussion at the start of this project. 



\section{Integer-valued Gromov--Witten invariants}\label{section2}

In this section, we review the genus zero integer-valued Gromov--Witten invariants constructed in \cite{Bai_Xu_2022}. The ordinary GW invariants of a general symplectic manifold fail to be integers because of the orbifold nature of the moduli spaces of stable maps. The integer-valued GW invariants originated from the insight of Fukaya--Ono \cite{Fukaya_Ono_integer}, where they observed the possibility of using a particular kind of single-valued perturbations on Kuranishi structures on the moduli spaces of stable maps. A crucial point is that the moduli spaces carry (stable) complex structures. On the technical level the idea was further explored by B. Parker \cite{BParker_integer} who clarified crucial properties needed for the special kind of perturbations Fukaya--Ono proposed. Eventually, the first and the third author \cite{Bai_Xu_2022} polished the ideas of Fukaya--Ono and B. Parker and rigorously constructed the ${\mb Z}$-valued GW invariants in genus zero.

\subsection{Derived orbifold charts and FOP pseudocycle}

A derived orbifold chart (D-chart for short) is a triple
\beqn
{\mc C} = ({\mc U}, {\mc E}, {\mc S})
\eeqn
where ${\mc U}$ is a smooth orbifold, ${\mc E} \to {\mc U}$ is an orbifold vector bundle (known as the obstruction bundle), and ${\mc S}: {\mc U} \to {\mc E}$ is a continuous section such that ${\mc S}^{-1}(0)$ is compact. A morphism of D-charts from ${\mc C}_1$ to ${\mc C}_2$ consists of a map $\iota_{21}: {\mc U}_1 \to {\mc U}_2$ and a bundle map $\widehat \iota_{21}: {\mc E}_1 \to {\mc E}_2$ covering $\iota_{21}$ such that 
\beqn
\widehat\iota_{21} \circ {\mc S}_1 = {\mc S}_2 \circ \iota_{21}.
\eeqn
There are the following typical types of morphisms.
\begin{enumerate}
    \item Shrinking/open embedding induced by shrinking ${\mc U}$ to an open neighborhood of ${\mc S}^{-1}(0)$ and restricting ${\mc E}$ together with ${\mc S}$.

    \item Stabilization by a vector bundle. Namely, there is a vector bundle $\pi_{\mc F}: {\mc F} \to {\mc U}_1$ such that ${\mc U}_2$ is the total space of ${\mc F}$, ${\mc E}_2 = \pi_{\mc F}^* {\mc E}_1 \oplus \pi_{\mc F}^* {\mc F}$, and ${\mc S}_2 = \pi_{\mc F}^* {\mc S}_1 \oplus \tau_{\mc F}$ where $\tau_{\mc F}$ is the tautological section; moreover $\iota_{21}$ is the inclusion of the zero section and $\wh\iota_{21}$ is the embedding of ${\mc E}_1$ into the first summand of ${\mc E}_2$. 
\end{enumerate}

Moreover, a {\bf strict cobordism} from ${\mc C}_1$ to ${\mc C}_2$ is a D-chart with boundary $\tilde {\mc C} = (\tilde {\mc U}, \tilde {\mc E}, \tilde {\mc S})$ together with identifications
\beqn
\partial \tilde {\mc C} = (\partial \tilde {\mc U}, \tilde {\mc E}|_{\partial \tilde {\mc U}}, \tilde {\mc S}|_{\tilde {\mc U}}) \cong {\mc C}_1 \sqcup {\mc C}_2.
\eeqn

\begin{defn}\label{defn_chart_equivalence}
We say that two D-charts ${\mc C}_1$ and ${\mc C}_2$ are {\bf equivalent} if ${\mc C}_1$ can be connected to ${\mc C}_2$ by a zig-zag of open embeddings or stabilizations; ${\mc C}_1$ and ${\mc C}_2$ are called {\bf cobordant} if they are equivalent to a strictly cobordant pairs.\footnote{The relation induced by the shrinking morphisms is a special case of the cobordism relation. We keep the shrinking morphisms explicit because they naturally arise from transversality arguments in the geometric construction of D-charts for moduli spaces of $J$-holomorphic curves.} By \cite[Proposition 5.1]{pardon2020orbifold}, the cobordism relation defines an equivalence relation among D-charts.
\end{defn}

We expect the zero locus ${\mc S}^{-1}(0)$ to carry a fundamental class which is invariant under cobordism relations. This is the case in rational coefficients and the fundamental class is roughly the Poincar\'e dual of the Euler class of the obstruction bundle ${\mc E}$. Equivalently, a fundamental cycle can be obtained from the zero locus of a multi-valued transverse perturbation of ${\mc S}$. Fukaya--Ono's idea to obtain a fundamental class with integer coefficients is to carefully choose single-valued perturbations satisfying a more refined version of the transversality condition, although the topological nature of such a fundamental class is not completely clear.

To carry out the Fukaya--Ono--Parker (FOP) perturbation scheme, one needs a version of complex structures on orbifolds and the obstruction bundles. A {\bf normal complex structure} on an effective orbifold ${\mc U}$ associates to each orbifold chart $(U, \Gamma)$ and each subgroup $H \subset \Gamma$ an $H$-invariant complex structure on the normal bundle $NU^H \to U^H$ of the $H$-fixed point set $U^H \subset U$ such that orbifold coordinate changes respect these complex structures. If ${\mc E} \to {\mc U}$ is an orbifold vector bundle, a normal complex structure on ${\mc E}$ associates to each bundle chart $(U, E, \Gamma)$ (where $(U,\Gamma)$ is an orbifold chart and $E\to U$ is a $\Gamma$-equivariant vector bundle) and each subgroup $H\subset \Gamma$ an $H$-invariant complex structure on the subbundle $\check E^H \subset E|_{U^H}$ (which is the direct sum of nontrivial irreducible $H$-representations contained in $E|_{U^H}$) such that the bundle coordinate changes respect these complex structures. A normal complex structure on a D-chart ${\mc C} = ({\mc U}, {\mc E}, {\mc S})$ consists of a normal complex structure on ${\mc U}$ and a normal complex structure on ${\mc E}$. The notion of derived cobordism can be defined for normally complex D-charts if we require that stabilizations are via complex vector bundles.

\begin{thm}\cite{Bai_Xu_2022}\label{thm_FOP_property}
Given any normally complex and effective orbifold ${\mc U}$ and a normally complex orbifold vector bundle ${\mc E} \to {\mc U}$, there is a class of (single-valued) smooth sections, called {\bf FOP transverse sections} which satisfy the following conditions. 

\begin{enumerate}

\item The FOP transversality condition is local; moreover, over the isotropy-free part of $\mc U$, ${\mc U}_{\rm free} \subset {\mc U}$ (which is a manifold), being FOP transverse is equivalent to being transverse in the classical sense. 

\item The condition that a smooth section ${\mc S}$ is FOP transverse only depends on the behavior of ${\mc S}$ near ${\mc S}^{-1}(0)$. In particular, any smooth section ${\mc S}$ is FOP transverse away from ${\mc S}^{-1}(0)$. 

\item Given a continuous norm on ${\mc E}$, for any continuous section ${\mc S}_0: {\mc U} \to {\mc E}$ and any $\delta>0$, there exists an FOP transverse section ${\mc S}: {\mc U} \to {\mc E}$ such that
\beqn
\| {\mc S}_0 - {\mc S}\|_{C^0} \leq \delta.
\eeqn

\item {\bf (CUDV property)} Given any closed subset $C \subset {\mc U}$ and a smooth section ${\mc S}_0: O \to {\mc E}$ defined over an open neighborhood $O$ of $C$, if ${\mc S}_0$ is FOP transverse near $C$, then there exists an FOP transverse section ${\mc S}: {\mc U} \to {\mc E}$ which agrees with ${\mc S}_0$ near $C$.

\item {\bf (Stabilization property)} Suppose $\pi_{\mc F}: {\mc F} \to {\mc U}$ is an orbifold complex vector bundle and $\tau_{\mc F}: {\mc F} \to \pi_{\mc F}^* {\mc F}$ is the tautological section. If ${\mc S}: {\mc U} \to {\mc E}$ is an FOP transverse section, then the section
\beqn
\pi_{\mc F}^* {\mc S} \oplus \tau_{\mc F}: {\mc F} \to \pi_{\mc F}^* {\mc E} \oplus \pi_{\mc F}^* {\mc F}
\eeqn
is a also an FOP transverse section.

\item If ${\mc Z} \subset {\mc U}$ is a closed and proper sub-orbifold whose normal bundle is an ordinary vector bundle (i.e., the fibers as representations of stabilizers are a direct sum of trivial representations, which implies ${\mc Z}$ is also normally complex), and if ${\mc S}: {\mc Z} \to {\mc E}$ is an FOP transverse section, then there exists an FOP transverse extension of ${\mc S}$ to ${\mc U}$.

\item For any FOP transverse section ${\mc S}: {\mc U} \to {\mc E}$, the isotropy-free part of the zero locus
\beqn
{\mc S}^{-1}(0)_{\rm free}:= {\mc S}^{-1}(0) \cap {\mc U}_{\rm free}
\eeqn
is a transverse intersection in the classical sense. Moreover, its boundary
\beqn
\ov{{\mc S}^{-1}(0)_{\rm free}} \setminus {\mc S}^{-1}(0)_{\rm free} \subset {\mc U} \setminus {\mc U}_{\rm free}
\eeqn
is the union of images of smooth maps from manifolds of dimension at most ${\rm dim}{\mc U} - {\rm rank}{\mc E} - 2$.
\end{enumerate}
\end{thm}

Now if we are given a normally complex effective derived orbifold chart $({\mc U}, {\mc E}, {\mc S})$, we can choose an FOP transverse section ${\mc S}'$ which agrees with ${\mc S}$ outside a compact neighborhood of ${\mc S}^{-1}(0)$. Then $({\mc S}')^{-1}(0)$ is compact and the isotropy-free part
\beqn
({\mc S}')^{-1}(0)_{\rm free}
\eeqn
is a pseudocycle \footnote{The notion of pseudocyle in orbifolds is defined in \cite{Bai_Xu_2022}.} of dimension ${\rm dim} {\mc U} - {\rm rank} {\mc E}$. Hence if ${\mc U}$ and ${\mc E}$ are oriented, this pseudocycle represents an integral homology class in ${\mc U}$. Moreover, given any two FOP transverse perturbations ${\mc S}_1', {\mc S}_2'$, the two pseudocycles $(S_1')^{-1}(0)_{\rm free}$ and $(S_2')^{-1}(0)_{\rm free}$ are cobordant. Therefore the integral homology class, which we call the FOP Euler class
\beqn
\chi^{\rm FOP}({\mc C}) \in H_*({\mc U}; {\mb Z}),
\eeqn
is well-defined. 

One typically needs to push forward the homology class into another space. To connect with classical cobordism theory, we restrict to the case when the derived orbifold charts are {\it stably complex}, meaning that the virtual vector bundle $T{\mc U} - {\mc E}$ has a stable complex structure (see \cite[Definition 6.11, 6.14]{Bai_Xu_2022}). In particular, if $T{\mc U}$ and ${\mc E}$ are both complex vector bundles, then ${\mc C} = ({\mc U}, {\mc E}, {\mc S})$ is automatically stably complex and normally complex. In the special case when ${\mc U}$ is a manifold and ${\mc S}: {\mc U} \to {\mc E}$ is transverse, then a stable complex structure on $({\mc U}, {\mc E}, {\mc S})$ induces a stable complex structure on the manifold ${\mc S}^{-1}(0)$. On the other hand, if a derived orbifold chart ${\mc C}$ is stably complex, then via a stabilization, it is equivalent to a normally complex one.

For any topological space $X$, the stably complex derived orbifold bordism group 
\beqn
\Omega^{{\rm der}, {\mb C}}_k(X)
\eeqn
is the abelian group generated by isomorphism classes of quadruples $({\mc U}, {\mc E}, {\mc S}, f)$ where $({\mc U}, {\mc E}, {\mc S})$ is an stably complex D-chart of virtual dimension ${\rm dim} {\mc U} - {\rm rank} {\mc E} = k$ and $f: {\mc U}\to X$ is a continuous map, modulo the equivalence relation induced from stabilization by complex vector bundles and cobordism respecting the stable complex structures. The assignment $\Omega^{{\rm der}, {\mb C}}_{*}(-)$ actually defines a generalized homology theory. Then the pushforward of the FOP pseudocycle (and its homology class) induces a natural transformation of generalized homology theories
\beqn
\Omega^{{\rm der}, {\mb C}}_{*}(-) \to H_{*}(-; {\mb Z}).
\eeqn

In fact, when the target space is a manifold, this natural transformation factors through pseudocycles. As we do not know if pseudocycles up to cobordism is a homology theory or not, we only consider the naive properties. Namely, for any manifold $X$, one has a group homomorphism
\beq\label{eqn:FOP-transform}
\Omega_k^{{\rm deg}, {\mb C}}(X) \to {\mc H}_k(X)
\eeq
where ${\mc H}_k(X)$ is the abelian group of $k$-dimensional oriented pseudocycles up to cobordism (see Appendix).

\subsection{AMS charts and the integer-valued GW invariants}

Let $(X, \omega)$ be a compact symplectic manifold and $J$ be an $\omega$-compatible almost complex structure. Given $k \geq 0$ and $A \in H_2(X; {\mb Z})$, consider the moduli space of stable $J$-holomorphic spheres with $k$ marked points in class $A$, denoted by 
\beqn
\ov{\mc M}_{0, k}(X, J, A).
\eeqn
There is a continuous evaluation map 
\beqn
\ev: \ov{\mc M}_{0, k}(X, J, A) \to X^k.
\eeqn
A {\bf D-chart lift} of $(\ov{\mc M}_{0, k}(X, J, A), \ev)$ is a tuple
\beqn
({\mc U}, {\mc E}, {\mc S}, \psi, \wt \ev)
\eeqn
where $({\mc U}, {\mc E}, {\mc S})$ is a D-chart, $\psi$ is an isomorphism of orbispaces from ${\mc S}^{-1}(0)$ to $\ov{\mc M}_{0, k}(X, J, A)$, and $\wt\ev: {\mc U} \to X^k$ is an extension of the evaluation map to ${\mc U}$. In \cite{AMS, AMS2} (see also \cite{Hirschi_Swaminathan}) a class of global Kuranishi charts (see Definition \ref{defn_Kchart}) on $\ov{\mc M}{}_{0, k}(X, J, A)$ were constructed, inducing a class of derived orbifold charts lifts (which we call AMS D-charts). Moreover, for such a chart, one can make sure that both $T{\mc U}$ and $T{\mc E}$ are complex vector bundles. In particular, the AMS D-charts are stably complex. The construction of an AMS chart depends on various choices but the equivalence class does not. Hence the quadruple $({\mc U}, {\mc E}, {\mc S}, \tilde \ev)$ induces an element of $\Omega_{ i(A)}^{{\rm der}, {\mb C}}(X^k)$ where $i(A)$ is the expected dimension of $\ov{\mc M}{}_{0, k}(X, J, A)$. Moreover, if $(\omega, J)$ can be deformed to a pair $(\omega', J')$ via a smooth path $(\omega_t, J_t)$, then the resulting AMS charts are derived cobordant. Therefore the corresponding element in $\Omega_{ i(A)}^{{\rm der}, {\mb C}}(X^k)$, only depends on the deformation class of $\omega$. By applying the map \eqref{eqn:FOP-transform}, we obtain a cobordism class of oriented integral pseudocycles
\beqn
\GW^{\mb Z}_{0, k}(X, A)^{\rm vir} \in {\mc H}_k(X).
\eeqn

\subsubsection{Reduction to characteristic $p$}

For the purpose of this article, we would like to reduce the integral cycle and invariants to characteristic $p$. Let ${\mb k}$ be a field. For any $n$-dimensional compact oriented manifold $Y$, there holds the Poincar\'e duality 
\beqn
H_k (Y; {\mb k}) \cong {\rm Hom}(H^k(Y; {\mb k}), {\mb k}).
\eeqn
Abusing the notation, the cobordism class of the integral Gromov--Witten pseudocycle induces a well-defined integral homology class
\beqn
\GW_{0, k}^{\mb Z}(X, A)^{\rm vir} \in H_{i(A)}(X^k; {\mb Z}).
\eeqn
By the discussion in Appendix \ref{app:A}, the associated integral pseudocycle also defines a pseudocycle with characteristic $p$ coefficients, which induces a homology class with $\mathbb{k}$-coefficient
\beqn
\GW_{0, k}^{\mb k}(X, A)^{\rm vir} \in H_{i(A)}(X^k; {\mb k}).
\eeqn
Hence we can define the correlator in ${\mb k}$-coefficients as 
\beqn
\GW^{\mb k}_{0, k} (\alpha_1, \ldots, \alpha_k) = \GW^{\mb k}_{0, k}(X, A)^{\rm vir} \cap (\alpha_1 \times \cdots \times \alpha_k),\ \forall \alpha_1, \ldots, \alpha_k \in H^*(X; {\mb k}).
\eeqn

\begin{rem}

\begin{enumerate}
    \item When ${\rm char}({\mb k}) = p$, even when all $\alpha_i$ come from integral cohomology, the above ${\mb k}$-valued GW invariants are not necessarily the mod $p$ reductions of the corresponding ${\mb Z}$-valued GW invariants, as $\alpha_i$ might be a $p$-torsion class in integral cohomology. In this situation, the correlator over ${\mb Z}$ has to vanish.

    \item When ${\mb k} = {\mb Q}$, the above invariants do not necessarily match the ordinary GW invariants. They do match when the symplectic manifold $(X, \omega)$ is semi-positive.
    
\end{enumerate}
\end{rem}

\section{Seidel maps}\label{section3}

In this section, we define a version of Seidel representations in finite characteristic. First recall the original construction. Given a loop of Hamiltonian diffeomorphisms $\phi = (\phi_t)_{t \in S^1}$ on a compact symplectic manifold $(M, \omega_M)$, one can construct a Hamiltonian fibration 
\beqn
\pi_\phi: \tilde M_\phi \to S^2
\eeqn
with fibers being $M$ and clutching function along the equator being $\phi_t$. Endowing $S^2$ with the standard complex structure $J_{S^2}$, we can choose an almost complex structure $\tilde J_\phi$ on $\tilde M_\phi$ such that the projection $\pi_\phi$ is pseudo-holomorphic, the vertical tangent spaces are invariant under the action of $\tilde J_\phi$, and the fiberwise almost complex structure is tamed by the fiberwise symplectic form. One can consider the moduli space of $\tilde J_{\phi}$-holomorphic graphs, namely sections
\beqn
\tilde u: S^2 \to \tilde M_\phi
\eeqn
such that 
\beqn
\tilde J_\phi \circ d \tilde u = d\tilde u \circ J_{S^2}.
\eeqn
Let $\pi_2^{\rm graph} (\tilde M_\phi)$ be the set of homotopy classes of sections of $\tilde M_\phi$, i.e., elements in $\pi_{2}(\tilde{M}_{\phi})$ whose composition with $(\pi_{\phi})_*$ represent the generator of $\pi_{2}(S^2)$. For each $\tilde A \in \pi_2^{\rm graph} (\tilde M_\phi)$, let 
\beqn
\ov{\mc M}{}_{0,2}^{\rm graph} (\tilde J_{\phi}, \tilde A)
\eeqn
be the moduli space of stable $\tilde J_{\phi}$-holomorphic graphs of class $\tilde A$ with two marked points that are fixed to be $z_0 = 0, z_\infty = \infty$ in the domain $S^2$. The way to compactify the graph moduli space is to view the bubbles as holomorphic spheres in the fibers. Then there are evaluation maps
\beqn
\ev:= \ev_0 \times \ev_\infty: \ov{\mc M}{}_{0,2}^{\rm graph} (\tilde J_{\phi}, \tilde A) \to \tilde M_\phi|_{z_0} \times \tilde M_\phi|_{z_\infty} \cong M \times M.
\eeqn

If we assume that $(M, \omega_M)$ satisfies the $\mathbf{W}^+$ condition, a condition slightly stronger than semi-positivity (cf. \cite{seidel-rep}), then one can choose $\tilde J$ such that $\ev$ defines a pseudocycle. Using isomorphisms
\beqn
H^*(\tilde M_\phi|_{z_0}; {\mb Z}) \cong H^*(M; {\mb Z}) \cong H^* (\tilde M_\phi|_{z_\infty}; {\mb Z}),
\eeqn
and the intersection numbers between transverse pseudocycles, one can define the graph Gromov--Witten invariants 
\beqn
\GGW_{0, 2}^{\tilde M_\phi, \tilde A}(\alpha_0, \alpha_\infty) \in {\mb Z},\ \alpha_0, \alpha_\infty \in H^*(M; {\mb Z})
\eeqn
which is independent of the choice of almost complex structures and pseudocycle representatives: see the original construction \cite{seidel-rep}.

If we drop Seidel's ${\bf W}^+$ condition on $M$, one can construct a graph virtual cycle with rational coefficients (cf. \cite{mcduffseidel}), leading to ${\mb Q}$-valued graph Gromov--Witten invariants. These GW invariants can be made into a generating series which leads to an invertible element in quantum cohomology of $M$, and from different homotopy classes for Hamiltonian loops, they can be shown to define a group homomorphism
\beqn
\pi_{1}(\mathrm{Ham}(M, \omega_M)) \to QH^*(M)^{\times},
\eeqn
known as the Seidel representation.

In this paper, we will consider a variant of the graph moduli space. Suppose ${\bf z} = (z_1, \ldots, z_k)$ is a $k$-tuple of distinct points of $S^2 \setminus \{0, \infty\}$. One can consider the moduli space 
\beqn
{\mc M}{}_{0,2, {\bf z}}^{\rm graph} (\tilde J_{\phi}, \tilde A)
\eeqn
of $\tilde J_\phi$-holomorphic graphs $u: S^2 \to \tilde M_\phi$ with fixed marked $k+2$ marked points, ${\bf z}$ together with $z_0 = 0$ and $z_\infty = \infty$. Its compactification, denoted by $\ov{\mc M}{}_{0,2, {\bf z}}^{\rm graph} (\tilde J_{\phi}, \tilde A)$, is not exactly identical to $\ov{\mc M}{}_{0,2}^{\rm graph} (\tilde J_{\phi}, \tilde A)$. There is still the evaluation map 
\beqn
\ev:= \ev_0 \times \ev_\infty: \ov{\mc M}{}_{0,2, {\bf z}}^{\rm graph} (\tilde J_{\phi}, \tilde A) \to \tilde M_\phi|_{z_0} \times \tilde M_\phi|_{z_\infty} \cong M \times M
\eeqn
which pushes forward an integral virtual fundamental cycle that we will construct shortly.

\subsection{Global Kuranishi charts for graph moduli spaces}

To define graph Gromov--Witten invariants for general symplectic symplectic manifolds in arbitrary field coefficients, we would like to adapt the Abouzaid--McLean--Smith construction of global Kuranishi charts to the graph moduli spaces.

\begin{prop}\label{prop:d-chart}
There is a derived orbifold chart ${\mc C} = ({\mc U}, {\mc E}, {\mc S})$ with stable complex structure and a submersion
\beqn
\tilde{\ev}: {\mc U} \to M_0 \times M_\infty.
\eeqn
such that $({\mc U}, {\mc E}, {\mc S}, \tilde{\ev})$ is a D-chart lift of $(\ov{\mc M}{}_{0,2, {\bf z}}^{\rm graph} (\tilde J_{\phi}, \tilde A), \ev)$.
\end{prop}

Our construction resembles that of Abouzaid--McLean--Smith \cite{AMS2}. That is, we first construct a global Kuranishi chart (defined below) and then reduce to a derived orbifold chart.

\begin{defn}\label{defn_Kchart}
Let ${\mc M}$ be a topological space. 
\begin{enumerate}

\item A smooth (resp. topological) {\bf global Kuranishi chart} for ${\mc M}$ is a tuple 
\beqn
K = (G, V, E, S, \Psi)
\eeqn
where $G$ is a compact Lie group, $V$ is a smooth (resp. topological) $G$-manifold with at most finite isotropy groups, $E \to V$ is a smooth (resp. topological) $G$-equivariant vector bundle, $S: V \to E$ is a continuous $G$-equivariant section, and $\Psi: S^{-1}(0)/G \to {\mc M}$ is a homeomorphism.

\item An {\bf almost complex structure} on a smooth global Kuranishi chart $K = (G, V, E, S, \Psi)$ consists of a $G$-invariant complex structure on the bundle $TV \oplus {\mf g}$ and a $G$-invariant complex structure on $E$.

\item A {\bf singular global Kuranishi chart} over ${\mc M}$ is a tuple $K$ satisfying the same conditions for topological global Kuranishi chart as above except that $V$ is not required to be a topological manifold. $K$ is said to be regular/smooth over a $G$-invariant open subset $V' \subset V$ if $V'$ is a topological/smooth $G$-manifold. $K$ is said to be almost complex over a $G$-invariant open subset $V' \subset V$ if $(G, V', E|_{V'}, S|_{V'}, \Psi|_{V'})$ is a smooth global Kuranishi chart with an almost complex structure.
\end{enumerate}
\end{defn}

The notion of singular global Kuranishi charts will not be used until the proof of the main theorem in the next section.

One can see that given a smooth global Kuranishi chart as above, the quotient 
\beqn
{\mc C} = ({\mc U}, {\mc E}, {\mc S}, \psi) := (V/G, E/G, S/G, \Psi/G)
\eeqn
is a derived orbifold chart. Moreover, if the given global Kuranishi chart has an almost complex structure, the resulting derived orbifold chart is stably complex.

\subsubsection{Construction of global Kuranishi charts}
We describe how to construct an Abouzaid--McLean--Smith style global Kuranishi chart for the moduli space $\ov{\mc M}{}_{0,2, {\bf z}}^{\rm graph} (\tilde J_{\phi}, \tilde A)$.

First, we need to modify the topological energy. Upon choosing a Hamiltonian connection on $\tilde M_\phi$, one obtains a coupling form $\tilde \omega_\phi \in \Omega^2(\tilde M_\phi)$ which is closed and whose fiberwise restriction is the symplectic form of $M$. The cohomology class of $\tilde \omega_\phi$ in $\tilde M_\phi$ may not be rational. We find a rational approximation, i.e., a closed 2-form $\tilde\Omega_\phi \in \Omega^2(\tilde M_\phi)$ which is sufficiently close to $\tilde \omega_\phi$ and whose cohomology class is rational. One can guarantee that the fiberwise restrictions of the almost complex structure $\tilde J_\phi$ is still tamed by the fiberwise restrictions of $\tilde\Omega_\phi$. Then after an integral rescaling, one can assume that the cohomology class of $\tilde \Omega_\phi$ is integral.

Next, consider the moduli space of stable genus 0 holomorphic sections of the trivial fibration $\mathbb{CP}^d \times \mathbb{CP}^1 \to \mathbb{CP}^1$ whose projection to the $\mathbb{CP}^d$ factor represents $d$ times the generator of the second homology group, and we endow the parametrized domain with $k+2$ fixed marked points, two of which are identified with $0$ and $\infty$ and the other $k$ marked points are denoted by ${\bf z} = (z_1, \ldots, z_k)$, being a $k$-tuple of distinct points of $S^2 \setminus \{0, \infty\}$. This moduli space is denoted by
\beqn
\ov{\mc M}{}_{0,2, {\bf z}}^{\rm graph}(\mb{CP}^d, d).
\eeqn
Let $\tilde B_d \subset \ov{\mc M}{}_{0, 2, {\bf z}}^{\rm graph}(\mb{CP}^d, d)$ be the subset of curves whose images under the projection to $\mathbb{CP}^d$ are not contained in any hyperplane of $\mb{CP}^d$. Then $\tilde B_d$ is a smooth complex manifold with an action by $G_d: = U(d+1)$ via projective linear transformations. The universal curve
\beqn
\tilde C_d \to \tilde B_d
\eeqn
is naturally $G_d$-equivariant. For each $\rho \in \tilde B_d$, let $C_\rho \subset \tilde C_d$ be the domain of the fiber, which is a prestable genus zero curve with a parametrized componennt (called the {\bf principal component}). Consider smooth maps $u: C_\rho \to \tilde M_\phi$ such that after composing with $\tilde M_\phi \to S^2$, it contracts all components except that it identifies the principal components with the base $S^2$. To each such map $u$ there is an associated section
\beqn
\ov\partial_{\tilde J_\phi} u \in \Omega^{0,1}(C_\rho, u^* T^{\rm vert} \tilde M_\phi).
\eeqn

We follow \cite{AMS2} to thicken up the moduli space. Let $\tilde C_d^* \subset \tilde C_d$ be the complement of nodes and marked points, which is a smooth $G_d$-manifold. A {\bf finite-dimensional approximation scheme} on $\tilde C_d$ is a finite-dimensional complex $G_d$-representation $W$ and a $G_d$-equivariant linear map
\beqn
\iota: W \to C^\infty_c  \Big( \tilde C_d^* \times \tilde M_\phi, \Omega^{0,1}_{\tilde C_d^*/ \tilde B_d} \otimes T^{\rm vert} \tilde M_\phi \Big)
\eeqn
(where, in contrast to \cite{AMS2}, we do not require that the image of $\iota$ surjects onto the obstruction space). Then for each $e \in W$ and $\rho \in \tilde B_d$, one obtains a section 
\beqn
\iota (e)|_{C_\rho} \in C^\infty_c (C_\rho^* \times \tilde M_\phi, \Omega^{0,1}_{C_\rho} \otimes T^{\rm vert} \tilde M_\phi).
\eeqn
For any $u: C_\rho \to \tilde M_\phi$, the restriction to the graph of $u$
\beqn
\iota (e)(u) \in \Omega^{0,1}(C_\rho, u^* T^{\rm vert} \tilde M_\phi)
\eeqn
lies in the same space as where $\ov\partial_{\tilde J_\phi}(u)$ lives in. 

Now we define the so-called {\bf pre-thickening}. For $d = d(\tilde A) = \langle \tilde \Omega_\phi, \tilde A \rangle \in {\mb Z}$, consider triples $(\rho, u, e)$ where $\rho \in \tilde B_d$, $u: C_\rho \to \tilde M_\phi$ is a smooth graph, and $e \in W$ satisfying 1)
\beqn
\ov\partial_{\tilde J_\phi} u + \iota(e)(u) = 0,
\eeqn
and 2) the homotopy class of the graph $u$ is $\tilde A$. Denote by $\tilde V^{\rm pre}$ the moduli space of such triples endowed with the topology induced from the Hausdorff topology of closed subsets in $C_\rho \times \tilde M_\phi$ using the graph of $u$. There are natural maps 
\beqn
\tilde S^{\rm pre}: \tilde V^{\rm pre} \to W,\ \tilde S^{\rm pre}(\rho, u, e) = e
\eeqn
and 
\beqn
\tilde \Psi^{\rm pre}: (\tilde S^{\rm pre})^{-1}(0) \to \ov{\mc M}{}_{0, 2, {\bf z}}^{\rm graph}( \tilde J_\phi, \tilde A),\ (\rho, u, e) \to [C_\rho, u].
\eeqn

\begin{lemma}
There exists a finite-dimensional approximation scheme $(W, \iota)$ such that $\tilde V^{\rm pre}$ is a topological manifold near $(\tilde S^{\rm pre})^{-1}(0)$ of dimension $j(\tilde A) + {\rm dim} \tilde B_d + {\rm dim} W$ where 
\beqn
j (\tilde A ) = {\rm dim} M + 2 c_1(T^{\rm vert} \tilde M_\phi) (\tilde A)
\eeqn
is the expected dimension of the moduli space of holomorphic graphs in the class $\tilde A$. 
\end{lemma}

\begin{proof}
\cite[Lemma 4.2]{AMS2} guarantees the existence of finite-dimensional approximation schemes which can eliminate all cokernels of the linearized operators.
\end{proof}

Next we consider the analogue of framings on the graph moduli space. Any smooth section $u: C_\rho \to \tilde M_\phi$ of homotopy class $\tilde A$ pulls back a 2-form $u^* \tilde \Omega_\phi \in \Omega^2(C_\rho)$ whose integral over each component is an integer. Then there exists a unique Hermitian holomorphic line bundle (up to unitary isomorphism) on $C_\rho$, denoted by $L_u \to C_\rho$ whose curvature form is $- 2\pi {\bf i} u^* \tilde \Omega_\phi$. The total degree of $L_u$ is $d = d(\tilde A)$. A {\bf framing} over $u$ is a basis
\beqn
F = (f_0, \ldots, f_d)
\eeqn
of the space $H^0(L_u)$. For each framing $F$, define a Hermitian matrix $H(F)$ with entries
\beqn
H_{ij}(F):= \int_{C_\rho} \langle f_i, f_j \rangle u^* \tilde \Omega_\phi.
\eeqn
Each framing (which must be base point free) also defines a holomorphic map 
\beqn
\rho_F: C_\rho \to \mb{CP}^1 \times \mb{CP}^d,
\eeqn
whose projection to the $\mb{CP}^1$-component contracts all components except for mapping the principal component isomorphically to $\mb{CP}^1$, and the projection to the $\mb{CP}^d$ component is given by $\rho_F(z) = [f_0(z), \ldots, f_d(z)]$. In particular, $\rho_F$ defines an element in $\ov{\mc M}{}_{0,2, {\bf z}}^{\rm graph}(\mb{CP}^d, d)$.

Let $\tilde V$ be the space of quadruples $(\rho, u, e, F)$ where $(\rho, u, e) \in \tilde V^{\rm pre}$ and $F$ is a framing over $u$. This is the base of the AMS chart. Notice that $\tilde V \to \tilde V^{\rm pre}$ has the structure of a $G_d$-equivariant principal $G_d^{\mb C} = GL(d+1)$-bundle where $GL(d+1)$ transform the framing linearly. Moreover
\beqn
{\rm dim} \tilde V = j(\tilde A) + {\rm dim} W + {\rm dim} \tilde B_{d} + 2 {\rm dim} G_d.
\eeqn

The obstruction bundle of the AMS chart is the direct sum
\beqn
\tilde E = W \oplus \pi_{\tilde V/\tilde B}^* T\tilde B_d \oplus  {\mf g}_d
\eeqn
where $\pi_{\tilde V/\tilde B}: \tilde V \to \tilde B_{d}$ is the forgetful map $(\rho, u, e, F) \mapsto \rho$, while $W$ and ${\mf g}_d$ denote the trivial bundle with fibers $W$ and the Lie algebra ${\mf g}_d$ respectively. One can see
\beqn
{\rm dim} \tilde V - {\rm rank} \tilde E - {\rm dim} G_d = j(\tilde A) = {\rm dim}^{\rm vir} \ov{\mc M}{}_{0,2, {\bf z}}^{\rm graph} (\tilde J_{\phi}, \tilde A).
\eeqn
To define the Kuranishi section $\tilde S: \tilde V \to \tilde E$, we choose a $G_d$-invariant Riemannian metric on $\tilde B_{d}$ so that the exponential map $\exp_{\tilde B_d}$ identifies an open neighborhood $\Delta^+(\tilde B_{d})$ of the diagonal $\Delta (\tilde B_{d}) \subset \tilde B_{ d} \times \tilde B_{ d}$ with a neighborhood of the zero section of the tangent bundle $T\tilde B_{d}$. By abuse of notations, shrink $\tilde V$ to the open subset of quadruples $(\rho, u, e, F)$ such that $(\rho, \rho_F)\in \Delta^+(\tilde B_d)$ and such that the Hermitian matrix $H_F$ is invertible. Denote this open subset still by $\tilde V$. Then define 
\beqn
\begin{aligned}
\tilde S: \tilde V &\to \tilde E \\
(\rho, u, e, F) &\mapsto \Big( e, \exp_{\tilde B_d}^{-1}(\rho, \rho_F), \exp_H^{-1} (H(F)) \Big).
\end{aligned}
\eeqn
Here $\exp_H$ is the exponential map of matrices; $\exp_H^{-1}(H(F)) = 0$ is equivalent to that $F$ is a unitary basis. The AMS global chart for $\ov{\mc M}{}_{0, 2}^{\rm graph}(\tilde M_\phi, \tilde A)$ is the tuple 
\beq\label{eqn:k-chart}
\tilde K = (G_d, \tilde V, \tilde E, \tilde S, \tilde \psi)
\eeq
where $\tilde \psi$ is the natural map
\beqn
\tilde \psi: (\tilde S)^{-1}(0) / G_d \to \ov{\mc M}{}_{0,2, {\bf z}}^{\rm graph} (\tilde J_{\phi}, \tilde A).
\eeqn

\begin{lemma}\label{lemma34}
The map $\tilde \psi$ is a homeomorphism.
\end{lemma}

\begin{proof}
This is the same as the case of \cite{AMS2}. We first prove that $\tilde \psi$ is a bijection. Given any point of $\ov{\mc M}{}_{0,2}^{\rm graph}(\tilde J_\phi, \tilde A)$ represented by a stable graph $u: C \to \tilde M_\phi$ where $C$ is a smooth or nodal genus zero curve with a parametrized principal component, consider the pullback bundle $L_u \to C$. Then as $\tilde J_\phi$ is tamed by $\tilde \Omega_\phi$, the pullback form $u^* \tilde \Omega_\phi$ is positive on effective components of $C$. Consider any unitary framing $F$ of $L_u$ with respect to the $L^2$ inner product induced from $u^* \Omega_\phi$. Then $F$ induces a stable map $\rho_F: C \to \mb{CP}^1 \times \mb{CP}^d$ whose projection to the second component is of degree $d = d(\tilde A)$, hence represents a point of $\tilde B_d$. One can then identify $u$ with a holomorphic graph $u_F: C_{\rho_F} \to \tilde M_\phi$. Therefore, 
\beqn
( \rho_F, u_F, 0, F) \in \tilde S^{-1}(0)
\eeqn
and is sent by $\tilde \psi$ to the point represented by the given $u$. Therefore $\tilde\psi$ is surjective. To show that $\tilde \psi$ is injective, suppose $\tilde\psi$ sends $(\rho_1, u_1, 0, F_1)$ and $(\rho_2, u_2, 0, F_2)$ to the same point of the graph moduli. Then the two stable graphs $u_1: C_{\rho_1} \to \tilde M_\phi$ and $u_2: C_{\rho_2} \to \tilde M_\phi$ are equivalent. Therefore, there exists a biholomorphic map $\varphi: C_{\rho_1} \to C_{\rho_2}$ which is the identity on the principal component such that 
\beqn
u_1 = u_2 \circ \varphi.
\eeqn
Therefore, $\varphi^* u_2^* \tilde \Omega_\phi = u_1^* \tilde \Omega_\phi$, implying that $L_{u_1} = \varphi^* L_{u_2}$ and $\varphi^* F_2$ is a unitary framing of $L_{u_1}$. Therefore, there is an element $g \in G_d$ such that $F_1 = g \varphi^* F_2$. This implies that $(\rho_1, u_1, 0, F_1)$ and $(\rho_2, u_2, 0, F_2)$ are on the same $G_d$-orbit. Hence $\tilde \psi$ is injective. Lastly it is straightforward to see that $\tilde\psi$ is continuous and a homeomorphism because the Gromov topology on stable maps agrees with the $C^0$-topology.
\end{proof}

We also notice that there is a well-defined $G_d$-invariant evaluation map 
\beqn
\ev = (\ev_0, \ev_\infty): \tilde V \to M_0 \times M_\infty, \ev(\rho, u, e, F) = (u(0), u(\infty))
\eeqn
which extends the evaluation map on the moduli space. 

\subsubsection{Stable smoothing}

Next we need to put a smooth structure on the AMS Kuranishi chart. As the smoothing procedure will be carefully engineered in the proof of the main theorem (c.f. the proof of Proposition \ref{prop_chart_detail}), we provide more details here, despite the fact that for the situation treated in this section it is almost identical to the original argument of Abouzaid--McLean--Smith \cite{AMS}.

First we recall a few notions related to smoothing. Let $G$ be a compact Lie group and $Y$ be a topological $G$-manifold. A $G$-smoothing on $Y$ is a smooth structure on $Y$ such that the $G$-action is smooth. A {\bf stable $G$-smoothing} on $Y$ is a $G$-smoothing on the product $Y \times R$ where $R$ is a finite-dimensional representation and $G$ acts on $Y \times R$ diagonally.

Any topological $G$-manifold $Y$ has its $G$-equivariant tangent microbundle $T_\mu Y$ (see \cite{Milnor_micro_1}). A necessary condition for $Y$ to admit a $G$-smoothing is the existence of a $G$-equivariant vector bunde lift of $T_\mu Y$, i.e., a $G$-equivariant microbundle isomorphism 
\beqn
F_\mu \to T_\mu Y
\eeqn
where $F \to Y$ is a $G$-equivariant vector bundle and $F_\mu$ is the induced microbundle. The crucial ingredient in the smoothing procedure in \cite{AMS} is the stable $G$-smoothing theorem of Lashof \cite{Lashof_1979} which says that once the tangent microbundle $Y$ admits a $G$-equivariant vector bundle lift, then $Y$ admits a stable $G$-smoothing. In addition, there is a canonical one-to-one correspondence between stable $G$-isotopy classes of $G$-equivariant vector bundle lifts of $T_\mu Y$ and stable $G$-isotopy classes of stable $G$-smoothings of $Y$.

The technical result regarding smoothing the AMS chart is stated below.

\begin{lemma}\label{lemma_smoothing}
There exists a stable $G_d$-smoothing on $\tilde V$, or equivalently, a $G_d$-smoothing on a product $\tilde V \times R$ where $R$ is a $G_d$-representation. 
\end{lemma}

One can use the trivial bundle $\uds R \to \tilde V$ to stabilize the Kuranishi chart so the base of the new chart is smooth. On the other hand, equivariant continuous vector bundles are always isomorphic to equivariant smooth vector bundles. Hence the Kuranishi chart stabilized by $\uds R$ is smooth.

\begin{proof}[Proof of Lemma \ref{lemma_smoothing}]
The proof is essentially the same as the argument of \cite[Section 4]{AMS}. The space $\tilde V$ is a topological $G_d$-manifold. The projection
\beq\label{eqn:rel-smooth}
\pi_{\tilde V/\tilde B}: \tilde V \to \tilde B_d
\eeq
is a topological submersion (see \cite[Definition 4.18]{AMS}) whose fibers are canonically smooth manifolds. Moreover, the smooth structures on fibers vary sufficiently regularly (known as $C^1_{loc}$ $G$-bundles, see \cite[Section 4.5]{AMS} and \cite[Corollary 6.29]{AMS}) so that there is a vertical tangent bundle 
\beqn
T^{\rm vert} \tilde V \to \tilde V
\eeqn
whose restrictions to fibers of $\pi_{\tilde V / \tilde B}$ are the tangent bundles of the fibers. Moreover one can construct a $G_d$-equivariant vector bundle lift
\beq\label{eqn:vec-lift}
(T^{\rm vert} \tilde V \oplus \pi_{\tilde V/ \tilde B}^* T\tilde B_d)_\mu \to T_\mu \tilde V
\eeq
(see \cite[Corollary 4.26]{AMS}). Hence by the theorem of Lashof \cite{Lashof_1979}, there is a representation $R$ and a $G_d$-smoothing on $\tilde V \times R$.
\end{proof}

\begin{rem}\label{rem:smoothing}
We also remark that although the $G_d$-equivariant vector bundle lift is manually constructed, the stable $G_d$-isotopy class of the $G_d$-smoothing is canonical. In other words, up to a stabilization by a $G_d$-representation, any two stable smoothings of $\tilde{V}$ built from the vector bundle lift \eqref{eqn:vec-lift} are equivariantly isotopic.
\end{rem}

\begin{proof}[Proof of Proposition \ref{prop:d-chart}]
By Lemma \ref{lemma34} and Lemma \ref{lemma_smoothing}, after a stabilization and taking the quotient by $G_d$, the global Kuranishi chart \eqref{eqn:k-chart} gives rise to a derived orbifold chart of the moduli space $\ov{\mc M}{}_{0, 2, {\bf z}}^{\rm graph}(\tilde M_\phi, \tilde A)$. As explained in \cite[Corollary 4.31]{AMS2}, we see that this derived orbifold chart is stably complex. Finally, using \cite[Lemma 4.5]{AMS}, possibly after a further stabilization, we can make sure that the extension of the evaluation map to the ambient orbifold of the derived orbifold chart is a smooth submersion.
\end{proof}


\subsubsection{Invariance}\label{subsubsec:independence}
We briefly comment on the independence of the chart $({\mc U}, {\mc E}, {\mc S}, \tilde{\ev})$ from Proposition \ref{prop:d-chart} on the choices made in the construction.

The tuple $\tilde{K}$ from \eqref{eqn:k-chart} defines a global Kuranishi chart of $\ov{\mc M}{}_{0,2, {\bf z}}^{\rm graph}(\tilde M_\phi, \tilde A)$. It depends on the following auxiliary data:
\begin{enumerate}
    \item the integral $2$-form $\tilde{\Omega}_\phi$;
    \item the finite-dimensional approximation scheme $(W, \iota)$;
    \item the $G_d$-invariant Riemannian metric on $\tilde{B}_d$ in the definition of $\tilde{S}$.
\end{enumerate}
To construct the stable smoothing, one needs to choose:
\begin{enumerate}
    \item a vector bundle lift \eqref{eqn:vec-lift};
    \item a $G_d$-representation $R$ such that $\tilde{V} \times R$ admits a $G_d$-equivariant smooth structure;
    \item an identification of the stabilized vector bundle $E$ with a smooth $G_d$-equivariant vector bundle.
\end{enumerate}

By a now-standard doubly framed curve argument (cf. \cite[Section 6.10]{AMS} and \cite[Section 4.9]{AMS2}) and the independence properties of the $G_d$-equivariant smoothing discussed in Remark \ref{rem:smoothing}, we see that for different choices, the resulting derived orbifold charts $({\mc U}, {\mc E}, {\mc S})$ are equivalent. Furthermore, using \cite[Section 6.11]{AMS}, we conclude that the derived orbifold bordism class over $M \times M$ represented by $({\mc U}, {\mc E}, {\mc S}, \tilde{\ev})$ is independent of $\tilde{J}_\phi$ and all the data above.

\subsection{The Seidel map via FOP transverse perturbations}

Now we can use the FOP perturbation scheme to construct the Seidel map. For the graph moduli $\ov{\mc M}{}_{0, 2, {\bf z}}^{\rm graph}(\tilde J_\phi, \tilde A)$, let 
\beqn
\tilde {\mc C} = (\tilde {\mc U}, \tilde {\mc E}, \tilde {\mc S}, \tilde \psi)
\eeqn
be the derived orbifold chart obtained in Proposition \ref{prop:d-chart}. By choosing an FOP transverse perturbation of $\tilde {\mc C}$, via the evaluation map, one obtains a well-defined homology class in the product $M_0 \times M_\infty$ (with coefficient in the field ${\mb k}$). By pairing with cohomology classes one obtains the graph GW invariant
\beqn
\GGW{}_{0,2, {\bf z}}^{\tilde M_\phi, \tilde A} (\alpha_0, \alpha_\infty) \in {\mb k},\ \alpha_0, \alpha_\infty \in H^*(M; {\mb k} ).
\eeqn

We write the Seidel map as a generating series in the Novikov variable. Let the Novikov field be 
\beqn
\Lambda = \Big\{ \sum_{i=1}^\infty a_i q^{\lambda_i}\ |\ a_i \in {\mb k},\ \lim_{i \to \infty} \lambda_i = +\infty \Big\}.
\eeqn
Let $\tilde \Omega_\phi \in \Omega^2(\tilde M_\phi)$ be the coupling form associated to the Hamiltonian fibration $\tilde M_\phi$, whose cohomology class does not depend on a Hamiltonian connection. Then define a linear map 
\beqn
S_{\bf z} (\phi): H^*(M; \Lambda) \to H^*(M; \Lambda)
\eeqn
by 
\beqn
S_{\bf z} (\phi)(\alpha) = \sum_{\tilde A \in \pi_2^{\rm graph}(\tilde M_\phi)} q^{\tilde \Omega_\phi (\tilde A)} \sum_{\beta, \gamma} \GGW_{0,2, {\bf z}}^{\tilde M_\phi, \tilde A}(\alpha, e_\beta) g^{\beta \gamma} e_\gamma
\eeqn
where $\{e_\beta\}$ is a basis of $H^*(M; {\mb k})$ and $g^{\beta\gamma}$ is the inverse of the intersection matrix $g_{\beta\gamma} = \langle e_\beta, e_\gamma\rangle$.

We expect that the collection of linear transformations $S(\phi)$ (the version with $\mathbf{z}$ being empty) defines a representation of $\pi_1({\rm Ham}(M, \omega_M))$ into the invertibles of quantum cohomology defined by the FOP perturbaion scheme. However the proof of this statement (which is essentially equivalent to the associativity of quantum cup product) is beyond the current understanding of the FOP perturbation method. For the purpose of this paper we only prove the invertibility of each $S_{\bf z} (\phi)$ without fully verifying the homomorphism property.

\begin{thm}\label{thm_Seidel_representation}
The map $S_{\bf z} (\phi)$ only depends on the homotopy class of $\phi$ and the number of fixed marked points. Therefore, we can denote it by $S_k(\phi)$.
\end{thm}

\begin{proof}
Except for the discussion in Section \ref{subsubsec:independence}, we first need to remove the dependence on ${\bf z}$. Indeed, any two distinct tuples ${\bf z}$ and ${\bf z}'$ can be connected by a path. One can use standard methods construct a cobordism between the corresponding derived orbifold charts. On the other hand, a homotopy of $\phi$ can be extended to a homotopy of geometric data (almost complex structures etc.), and in turn, a cobordism of the derived orbifold charts of the moduli spaces. As the FOP pseudocycles are cobordism invariants, the induced map $S_{\bf z} (\phi)$ on cohomology is invariant under homotopies of $\phi$.
\end{proof}

\begin{thm}\label{thm_invertible}
The map $S_k (\phi)$ is invertible for all $\phi: S^1 \to {\rm Ham}(M, \omega_M)$.
\end{thm}

\begin{proof}
See Subsection \ref{subsection_invertibility}.
\end{proof}

Accordingly, we immediately have

\begin{cor}\label{cor_nondegenerate}
For each $\phi: S^1 \to {\rm Ham}(M, \omega_M)$, the bilinear pairing
\beqn
(\beta_0, \beta_\infty) \mapsto \langle \beta_0, \beta_\infty\rangle_{k, \phi}:= \sum_{\tilde A \in \pi_2^{\rm graph}(\tilde M_\phi)} q^{\tilde \Omega_\phi(\tilde A)} \GGW_{0,2, {\bf z}}^{\tilde M_\phi, \tilde A}(\beta_0, \beta_\infty) \in \Lambda
\eeqn
is nondegenerate. \qed
\end{cor}

\subsection{Proof of Theorem \ref{thm_invertible}}\label{subsection_invertibility}

We first prove that $S_k (1)$ is invertible, where $k\geq 0$ is any nonnegative integer and $1$ is the constant loop in ${\rm Ham}(M, \omega_M)$ at the identity Hamiltonian diffeomorphism. The corresponding fibration $\tilde M_1$ is the trivial product $M \times S^2$. The moduli space of constant pseudo-holomorphic graphs is cut out transversely, and its induced map on cohomology is the identity. So, the derived orbifold chart constructed in Proposition \ref{prop:d-chart} is a stabilization of a manifold, thus the FOP pseudocycle agrees with the pushforward of the fundamental class of the moduli space under the evaluation map. Hence $S_k (1)$ is the identity map plus a higher order term in terms of the valuation on the Nokikov field $\Lambda$. We do not know how to prove the vanishing of the higher order term which counts parametrized holomorphic spheres in $M$. Nonetheless, it follows that $S_k (1)$ is invertible.

\subsubsection{Floer and Morse flow categories}

The difficulty of proving the expected property of the Seidel map
\beqn
S_k (\phi) \circ S_l (\psi) = S_{k+l} (\phi \# \psi)
\eeqn
where $\phi \# \psi$ is the product loop in ${\rm Ham}(M, \omega_M)$ is that it is a nontrivial task to identify the contribution of nodal configurations using the FOP perturbation scheme. To bypass this difficulty, we use Floer and Morse models; the corresponding moduli spaces can be viewed as certain real blowups along the codimension two divisors defined by nodal curves via introducing real framings on the nodes.

Introduce the discrete monoid
\beqn
\Pi = \{ \omega_M(A)\ |\ A \in \pi_2(M)\}\subset {\mb R}.
\eeqn
Fix a nondegenerate 1-periodic family of Hamiltonian $H_t$ (reminder: this is a datum independent of the Hamiltonian loop $\phi_t$). The critical point of the action functional ${\mc A}_{H_t}$ are equivalence classes of capped 1-periodic orbits, where two cappings are regarded equivalent if their difference has vanishing symplectic area. Fix an $S^1$-family of $\omega_M$-compatible almost complex structures $J_t$. The Hamiltonian Floer flow category of the pair $(H_t, J_t)$, denoted by ${\mb F}^{\rm Floer}$, has objects being equivalence classes of capped 1-periodic orbits, while for any pair of objects $(p, q)$, the morphism space is the compactified moduli space $\ov{\mc M}{}_{pq}^{\rm Floer}$ of solutions to the Floer equation for $(H_t, J_t)$ connecting $p$ and $q$. There is a natural free $\Pi$-action on this flow category.

On the other hand, fix a Morse--Smale pair $(f, g)$ on $M$. The Morse flow category, denoted by $\mb{F}^{\rm Morse}$, has objects being pairs $(a, p)$ where $a \in \Pi$ and $p\in {\rm crit} f$, while for any pair $((a, p), (b, q))$, the morphism space is the moduli space $\ov{\mc M}{}_{pq}^{\rm Morse}$ of possibly broken unparametrized Morse trajectories connecting $p$ and $q$. One can choose the pair $(f, g)$ such that each morphism space is an oriented smooth compact manifold-with-corners such that the composition maps are smooth embeddings compatible with the orientations (see \cite{Wehrheim-Morse}).

The Morse flow category naturally induces a chain complex (over any coefficient ring) with differential defined by counting $0$-dimensional morphism spaces with signs. Let $\Lambda$ be the Novikov field with ground field ${\mb k}$. Then the chain complex derived from $\mb{F}^{\rm Morse}$ coincides with 
\beqn
CM_*(f, g) \otimes \Lambda.
\eeqn
The homology is equal to $H_*(M; \Lambda)$. For the purpose of this paper, we do not need to consider grading.

On the other hand, to define the Floer chain complex from the Floer flow category $\mb{F}^{\rm Floer}$, one must regularize the moduli spaces. In \cite{Bai_Xu_Arnold, Bai_Xu_Floer} the first and third author constructed a class of so-called Kuranishi lifts of the Floer flow category $\mb{F}^{\rm Floer}$ which admits stable complex structures\footnote{In fact there is one technical step called ``outer-collaring'' which modifies the original Floer flow category by extending the boundary and corner strata outward.}. By using the FOP perturbation scheme, one obtains a Floer complex (over any coefficient field ${\mb k}$), denoted by 
\beqn
CF_*(H_t, J_t; \Lambda).
\eeqn
The resulting homology, $HF_*(H_t, J_t)$, as a finite-dimensional $\Lambda$-vector space, is independent of $(H_t, J_t)$. A crucial fact we need in this paper, proved in the upcoming \cite{Bai_Xu_Floer}, is the following.

\begin{thm}\cite{Bai_Xu_Floer}\label{thm_equal_rank}
When ${\mb k}$ is a field, ${\rm dim}_\Lambda HF_*(M; \Lambda) = {\rm rank} H_*(M;\mathbb{k})$.
\end{thm}

Before we define various alternate versions of the Seidel map, we need to recall the general notion of flow bimodules. If ${\mb F}$ and ${\mb G}$ are two flow categories, then a bimodule ${\mb M}$ over $({\mb F}, {\mb G})$ consists of, for each pair $p \in {\rm Ob}{\mb F}$ and $q \in {\rm Ob}{\mb G}$, a stratified topological space $M_{pq}^{\mb M}$ whose boundary and corner strata, roughly speaking, come from degenerations of morphism spaces of ${\mb F}$ and ${\mb G}$. For example, the union of codimension 1 strata of $M_{pq}^{\mb M}$ is the union of products
\beqn
\bigcup_{p' \in {\rm Ob}{\mb F}} M^{\mb F}_{pp'} \times M_{p' q}^{\mb M} \cup \bigcup_{q' \in {\rm Ob}{\mb G}} M^{\mb M}_{pq'} \times M_{q'q}^{\mb G}.
\eeqn
Examples of flow bimodules include moduli spaces of Floer continuation maps or Piunikhin--Salamon--Schwarz (PSS) maps. 

Flow bimodules can be composed (see the definition in \cite{Bai_Xu_Floer} and the discussion in the abstract setting \cite[Section 4]{Abouzaid_Blumberg_2024}). For example, the composition of two flow bimodules corresponding to two continuation maps is the bimodule whose moduli spaces are those of ``two-stage'' continuations (whose top strata has one breaking). There is also a notion of homotopy between bimodules. For example, the above-mentioned composed bimodule corresponding to concatenations of two continuation maps is homotopic to the bimodule corresponding to a single continuation map with data coming from ``gluing.''

As one expects, a flow bimodule should induce a chain map. In \cite{Bai_Xu_2022}, we defined a chain map induced from the PSS construction
\beqn
\Psi^{\rm PSS}: CM_*(f, g) \otimes \Lambda \to CF_*(H_t, J_t; \Lambda)
\eeqn
and a chain map in the opposite direction (which we call the SSP construction)
\beqn
\Psi^{\rm SSP}: CF_*(H_t, J_t; \Lambda) \to CM_*(f, g)\otimes \Lambda.
\eeqn
In addition, we proved that (cf. \cite[Theorem C]{Bai_Xu_2022}) the composition 
\beqn
\Psi^{\rm SSP}\circ \Psi^{\rm PSS}: CM_*(f, g) \otimes {\mb k} \to CM_*(f, g) \otimes \Lambda
\eeqn
is invertible (but whether the other composition is invertible is currently unknown). It follows from Theorem \ref{thm_equal_rank} that on the homology level, for field coefficients, the PSS and SSP maps are isomorphisms of $\Lambda$-vector spaces. More generally, in the forthcoming \cite{Bai_Xu_Floer}, we construct a chain map induced from a continuation map (which we do not need here) between two choices of Floer data $(H_t, J_t)$ and $(H_t', J_t')$ which induces an isomorphism on homology.\footnote{This result already follows from the exisiting construction in \cite[Appendix A]{Abouzaid_Blumberg_2024} because the FOP perturbation scheme can be applied to the regularizations of flow bimodule constructed from the moduli spaces of continuation maps discussed in \emph{loc.cit.}, which induces the desired chain maps between Floer complexes defined by different data.}

\subsubsection{Alternate Seidel maps}

Now consider a bimodule over $(\mb{F}^{\rm Morse}, {\mb F}^{\rm Morse})$, denoted by $\mb{M}^{\rm Morse}_{\bf z} (\phi)$, where ${\bf z}$ is a $k$-tuple of distinct marked points on $S^2 \setminus \{0, \infty\}$. First, for each class $\tilde A \in \pi_2^{\rm graph}(\tilde M_\phi)$ and each pair of critical points $y_0, y_\infty$ of the Morse function $f: M \to {\mb R}$, consider the moduli space of pearly object, consisting of a downward gradient ray $\xi_0: (-\infty, 0] \to M$ starting from $y_0$, a $\tilde J_\phi$-holomorphic graph $u: S^2 \to \tilde M_\phi$ of class $\tilde A$, and a downward gradient ray $\xi_\infty: [0, +\infty) \to M$ ending at $y_\infty$, such that
\beqn
\xi_0(0) = u(z_0),\ \xi_\infty(0) = u(z_\infty)
\eeqn
(for which we fix identifications $M_0 \cong M \cong M_\infty$). One can compactify this moduli space by allowing breaking of gradient rays and bubbling of holomorphic spheres in the fibers of $\tilde M_\phi$, while keeping the fixed marking ${\bf z}$. Denote the moduli space by 
\beqn
\ov{\mc M}{}_{y_0 y_\infty, {\bf z}}^{\rm graph}(\tilde J_\phi, \tilde A).
\eeqn
Then to each pair $((a, p), (b, q))$ of objects of $\mb{F}^{\rm Morse}$, define 
\beqn
M_{(a, p)\ (b, q)\ {\bf z}}^{\rm Morse}(\phi)
\eeqn
to be the union of $\ov{\mc M}{}_{pq, {\bf z}}^{\rm graph}(\tilde J_\phi, \tilde A)$ for all $\tilde A \in \pi_2^{\rm graph}(\tilde M_\phi)$ with $\tilde \Omega_\phi(\tilde A) = b - a$. It is straightforward to see that this provides a bimodule over $(\mb{F}^{\rm Morse}, \mb{F}^{\rm Morse})$, denoted by $\mb{M}_{\bf z}^{\rm Morse}(\phi)$. 

To define a chain map, notice that one has already constructed a derived orbifold chart for each degree $\tilde A$ for the graph moduli space. One can choose a FOP transverse perturbation for each such derived orbifold chart such that the evaluation maps restricted to the zero locus (in the free part) is transverse to (un)stable manifolds of critical point. Hence one obtains (well-defined) counts
\beqn
n_{pq}(\tilde A) \in {\mb Z}.
\eeqn
Sending these integers to the field ${\mb k}$ and weight them by $q^{\tilde \Omega_\phi(\tilde A)}$, one obtains a chain map 
\beqn
S_{\bf z}^{\rm Morse}(\phi): CM_*(f, g)\otimes \Lambda \to CM_*(f, g)\otimes \Lambda.
\eeqn
Denote the induced map on homology still by 
\beqn
S_{\bf z}^{\rm Morse}(\phi): H_*(M; \Lambda) \to H_*(M; \Lambda).
\eeqn

\begin{lemma}\label{lem:morse-cycle} 
On homology level, $S_{\bf z}^{\rm Morse}(\phi)$ is independent of ${\bf z}$ thus only depends on the number of markings $k$. Hence we can denote it by $S^{\rm Morse}_k (\phi)$. Moreover, $S_k^{\rm Morse}(\phi) = {\rm PD} \circ S_k (\phi)\circ {\rm PD}$.
\end{lemma}

\begin{proof}
The independence on ${\bf z}$ is similar to the proof of Theorem \ref{thm_Seidel_representation}. For different choices with the same number of markings, the induced flow bimodule are homotopic to each other. Therefore, the induced maps on homology agree. Next we compare the Morse-theoretic Seidel map and the Seidel map constructed in the previous subsection. In integer coefficients, for any Morse cycle, the (weighted) union of unstable manifolds defines a pseudocycle (see \cite{Schwarz_equivalence}). The same argument can be applied to the case of $\mb{F}_p$ coefficients. On the other hand, the pairing between homology and cohomology with ${\mb F}_p$ coefficients can be identified with the count of transverse intersection points between representing $p$-pseudocycles. Hence the identity holds for $\mb{k} = \mb{F}_p$. For general finite field $\mb{k}$ this follows from the universal coefficient theorem.
\end{proof}

Now we define a bimodule $\mb{M}_{\bf z}^{\rm PSS}(\phi)$ over $(\mb{F}^{\rm Morse}, \mb{F}^{\rm Floer})$, and another bimodule $\mb{M}_{\bf z}^{\rm SSP}(\phi)$ over $(\mb{F}^{\rm Floer}, \mb{F}^{\rm Morse})$. When $\phi = 1$, they are the same as the PSS and SSP bimodules considered in \cite{Bai_Xu_Arnold, Bai_Xu_Floer}. For a general loop $\phi: S^1 \to {\rm Ham}(M, \omega_M)$, the defining equation (over a disc with a cylindrical end) just includes the twisting by the loop $\phi$ and the moduli space only has the additional ingredients of the fixed marking ${\bf z}$ on the domains. Therefore, the construction of an associated Kuranishi chart lift and the chain map is completely analogous to the special case when $\phi = 1$. Therefore, one obtains linear maps between homologies
\beqn
S_{\bf z}^{\rm PSS}(\phi): H_*(M; \Lambda) \to HF_*(M; \Lambda),\ S_{\bf z}^{\rm SSP}(\phi): HF_*(M; \Lambda) \to H_*(M; \Lambda).
\eeqn

\begin{prop}\label{prop_composition}
On the homology level, the maps $S_{\bf z}^{\rm PSS}(\phi)$ and $S_{\bf z}^{\rm SSP}(\phi)$ do not depend on ${\bf z}$ but only the number of markings $k$. Hence we can denote them by $S_k^{\rm PSS}(\phi)$ and $S_k^{\rm SSP}(\phi)$. Moreover,
\beqn
S_k^{\rm SSP}(\psi) \circ S_l^{\rm PSS}(\phi) = S_{k+l}^{\rm Morse} (\phi \# \psi),\ \forall \phi, \psi: S^1 \to {\rm Ham}(M, \omega_M).
\eeqn
\end{prop}

\begin{proof}
The left hand side can be viewed as the numerical output of composing two flow bimodules and the equality can be viewed as the homological consequence of a homotopy of bimodules. The proof, which relies on constructing Kuranishi lifts and FOP perturbations, is completely analogous to the case of \cite{Bai_Xu_Arnold} which proves the equality for $\phi = \psi = 1$. 
\end{proof}


\begin{prop}
For any $\phi: S^1 \to {\rm Ham}(M, \omega_M)$ and $k \geq 0$, the maps $S_k^{\rm PSS}(\phi): H_*(M; \Lambda) \to HF_*(M; \Lambda)$ and $S_k^{\rm SSP}(\phi): HF_*(M; \Lambda) \to H_*(M; \Lambda)$ are both invertible. 
\end{prop}

\begin{proof}
Proposition \ref{prop_composition} implies $S_k^{\rm SSP}(\phi^{-1}) \circ S^{\rm PSS}_l(\phi) = S_{k+l} (1)$ which is invertible. It implies that the PSS maps are injective and the SSP maps are surjective. By Theorem \ref{thm_equal_rank}, when the ground ring $R$ is a field, as $\Lambda$-vector spaces, $HF^*(M; \Lambda)$ is isomorphic to $H^*(M; \Lambda)$. Hence the PSS and the SSP maps are both invertible. 
\end{proof}

\begin{proof}[Proof of Theorem \ref{thm_invertible}]
By composing $S_k^{\rm SSP}(\phi)$ with $S^{\rm PSS}(1)$, it follows from Proposition \ref{prop_composition} again that $S_k^{\rm Morse} (\phi)$ is invertible. Therefore, by Lemma \ref{lem:morse-cycle}, we conclude that $S_k (\phi)$ is invertible.
\end{proof}

\section{Cohomological splitting}\label{section4}

We prove the main theorems of this paper assuming the follwoing technical statement, whose proof is deferred to Section \ref{section5}.

\begin{thm}\label{thm41}
Let $(B, \omega_B)$ be a compact symplectic manifold with a nonzero Gromov--Witten invariant
\beqn
\GW_{0, k+2}^{B, A}([{\rm pt}], [{\rm pt}], a_1, \ldots, a_k) \neq 0
\eeqn
for $a_1, \ldots, a_k \in H^*(B; {\mb Z})$ and a curve class $A \in H_2(B; {\mb Z})$. Moreover, suppose there exist a compatible almost complex structure $J_B$, points $p_0, p_\infty \in B$, and pseudocycle representatives 
\beqn
f_i: W_i \to B
\eeqn
of the homology classes ${\rm PD}_B(a_i)$ satisfying the following conditions.
\begin{enumerate}

\item Inside the set of simple curves ${\mc M}{}_{0, k+2}^{\rm simple}(J_B, A) \subset \ov{\mc M}{}_{0,k+2}(J_B, A)$ there is an open subset ${\mc M}_{0, k+2}^{\rm reg}(J_B, A)$ which is cut out transversely.

\item There holds
\begin{multline}\label{eqn40}
\ev \Big( \ov{\mc M}{}_{0,k+2}(J_B, A) \Big) \cap \Big( \{ p_0 \}\times \{p_\infty\}\times \prod_{i=1}^k \ov{f_i (W_i)} \Big) \\
= \ev \Big( {\mc M}{}_{0, k+2}^{\rm reg}(J_B, A) \Big) \cap \Big( \{ p_0 \}\times \{p_\infty\} \times \prod_{i=1}^k f_i(W_i) \Big)
\end{multline}
and the intersection on the right hand side is transverse. 
\end{enumerate}
In particular, the Gromov--Witten invariant $\GW_{0,k+2}^{B, A}([{\rm pt}], [{\rm pt}], a_1, \ldots, a_k)$ is an integer. Then for any Hamiltonian fibration $P \to B$ with fiber $(M, \omega_M)$ and any coefficient field ${\mb k}$ whose characteristic does not divide this Gromov--Witten invariant, there is an isomorphism of graded $\mb{k}$-vector spaces
\beqn
H^*(P; {\mb k}) \cong H^*(B; {\mb k}) \otimes_{\mb k} H^*(M; {\mb k}).
\eeqn
\end{thm}

\begin{proof}[Proof of Theorem \ref{thm:split}]
If $(B, \omega_B)$ is monotone, it is standard that for a generic compatible almost complex structure $J_B$, the moduli space ${\mc M}{}_{0,k+2}^{\rm simple}(J_B, A)$ is regular, the evaluation map restricted to ${\mc M}{}_{0,k+2}^{\rm simple}(J_B, A)$ is a pseudocycle, and the Gromov--Witten invariants are defined via counting transverse intersections of pseudocycles (see \cite[Chapter 6]{McDuff_Salamon_2004}). Thus Theorem \ref{thm:split} follows from Theorem \ref{thm41}.
\end{proof}


\subsection{Proof of Theorem \ref{thm:B}}

Let $B$ be a smooth and stably rational projective variety of complex dimension $n$, meaning that $B \times \mb{CP}^r$ is birational to $\mb{CP}^{n+r}$. We first notice that by Lalonde--McDuff's ``surjection lemma'' (\cite[Lemma 4.1]{Lalonde_McDuff_2003}), it suffices to prove for the case that $B$ is rational. In this case, there is a rational map $\mb{CP}^n \dashedrightarrow B$. By Hironaka's theorem on resolution of indeterminacies (see \cite[Main Theorem II]{Hironaka_resolution}), there is another smooth projective variety $\tilde B$ fitting in a commutative diagram
\beqn
\xymatrix{  & \tilde B \ar[ld]_{\pi_1} \ar[rd]^{\pi_2} & \\
         \mb{CP}^n \ar@{..>}[rr] & &   B }
         \eeqn
where $\pi_1$ is the compositions of blowups with smooth centers and $\pi_2: \tilde B \to B$ is a genuine (birational) morphism. 
Then $\pi_1$ identifies a Zariski open set of $\tilde B$ with $\mb{CP}^n \setminus Z$ where $Z$ is the union of subvarieties of complex codimension at least two. Let $A \in \pi_2(\mb{CP}^n)$ be the line class. As a general line avoids $Z$, there is a natural lift $A \in \pi_2(\tilde B)$. 

\begin{lemma}\label{lemma51}
Let $\ov{\mc M}{}_{0,2}(J_{\tilde B}, A)$ be the moduli space of genus zero $2$-marked stable maps in class $A$ and $\ev: \ov{\mc M}{}_{0,2}(J_{\tilde B}, A) \to \tilde B^2$ be the evaluation map. 
\begin{enumerate}

\item There is an open subset ${\mc M}{}_{0,2}^{\rm reg}(J_{\tilde B}, A) \subset {\mc M}{}_{0,2}^{\rm simple}(J_{\tilde B}, A)$ which is transverse.

\item For a general pair of points $(p, q) \in \tilde B$, the preimage $\ev^{-1}((p, q))$ is a single point contained in ${\mc M}{}_{0,2}^{\rm reg} (J_{\tilde B}, A)$ and the intersection is transverse.

\item $\GW_{0,2}^{\tilde B, A}([{\rm pt}], [{\rm pt}]) = 1$.
\end{enumerate}
\end{lemma}
         
\begin{proof}
This statement is an easy extension of \cite[Proposition 4.15]{Lalonde_McDuff_2003}. 
For any pair of points $(p, q) \in \tilde B \times \tilde B$, there is a unique line connecting $\pi_1(p)$ and $\pi_1(q)$ in $\mb{CP}^n$ which underlies a regular stable map from $\mb{P}^1$. As $\tilde B$ and $\mb{CP}^n$ are isomorphic over a Zariski open set, when $(p, q)$ is general, the line is unique and avoids the subset $Z \subset \mb{CP}^n$, hence corresponds to a curve in $\tilde B$ in class $A$. This provides the regular open subset ${\mc M}{}_{0, k+2}^{\rm reg}(J_{\tilde B}, A) \subset {\mc M}{}_{0,2}^{\rm simple}(J_{\tilde B}, A)$. For a fixed general $(p, q)$ this curve is the only one in $\ov{\mc M}{}_{0,2}(J_{\tilde B}, A)$  passing through $p$ and $q$.
\end{proof}

\begin{cor}\label{cor52}
For any Hamiltonian fibration $\tilde P \to \tilde B$ and any coefficient field ${\mb k}$, the cohomology $H^*(\tilde P; {\mb k})$ splits.
\end{cor}

\begin{proof}
It follows from Lemma \ref{lemma51} that we are  in a special situation of Theorem \ref{thm41}. As the GW invariant is 1, one concludes the cohomological splitting for any Hamiltonian fibration over $\tilde B$ in all characteristic.
\end{proof}

\begin{proof}[Proof of Theorem \ref{thm:B}]
Now we derive the cohomological splitting property of $B$ from that of $\tilde B$. Remember that there is a birational morphism $\pi_2: \tilde B \to B$. Then by \cite[Lemma 2.6]{Bai_Pomerleano_2024}, because any Hamiltonian fibration over $\tilde B$ satisfies cohomological splitting over any field because of Corollary \ref{cor52}, it also holds for $B$.
\end{proof}

\begin{rem}  
The main theorem of \cite{ghs03} asserts that for a dominant morphism of varieties $P \to B$, if both $B$ and the general fiber are rationally connected, then $P$ is also rationally connected. On the other hand, the topological arguments in \cite{Lalonde_McDuff_2003} show that for a fiber bundle $M \hookrightarrow P \to B$, if every Hamiltonian fibration over $B$ and $M$ satisfies cohomological splitting, then the same holds for $P$. This showcases further potential relation between rational connectedness and cohomological splitting.
\end{rem}

\section{Proof of Theorem \ref{thm41}}\label{section5}

Now we start to prove the main technical result of this paper. First we need to make a notational modifications related to the topological type of rational curves. The Hurewicz map 
\beqn
\pi_2(B) \to H_2(B; {\mb Z})
\eeqn
may not be injective. Therefore, {\it a priori} a Gromov--Witten invariant of $B$ with a fixed homological degree may come from different homotopy classes of holomorphic spheres. However, it is straightforward to see that for each homotopy class of holomorphic sheres, the corresponding Gromov--Witten invariant is well-defined. Therefore, the non-vanishing condition of the GW invariants in a fixed homological degree in Theorem \ref{thm:split} implies that for some homotopy class, the invariant is nonzero and not divisible by ${\rm char}({\mb k})$. Hence without loss of generality, from now on the curve class $A$ labelling Gromov--Witten invariants of the base $B$ are homotopy classes rather than homology classes.

We would like to relate Gromov--Witten invariants of the base and Gromov--Witten invariants of the total space. We first consider the correspondence of between homotopy classes of spheres. 

\begin{lemma}
The natural map $(\pi_P)_*: \pi_2( P ) \to \pi_2(B)$ is surjective.
\end{lemma}

\begin{proof}
Suppose $A \in \pi_2(B)$ is represented by a smooth map $u: S^2 \to B$. The pullback fibration $u^* P \to S^2$ is then a Hamiltonian fibration. Let $\{\phi_t\}_{t\in S^1}$ be clutching function, which is a loop of Hamiltonian diffeomorphisms. Now consider the long exact sequence 
\beqn
\xymatrix{\cdots \ar[r]  & \pi_2(u^* P) \ar[r] & \pi_2(S^2) \ar[r] & \pi_1(X) \ar[r] & \cdots}
\eeqn
The map $\pi_2(S^2) \to \pi_1(X)$ sends the generator of $\pi_2(S^2)$ to the homotopy class of the loop $\phi_t(x)$ for any point $x$, which is always trivial (see \cite[Corollary 9.1.2]{McDuff_Salamon_2004}). Therefore, the map $\pi_2(u^* P) \to \pi_2(S^2)$ is surjective. It follows that the class $A$ lies in the image of $(\pi_P)_*$. 
\end{proof}

\subsection{The non-vanishing result of ${\mb k}$-valued GW in the total space}

The main object to consider is the ${\mb k}$-valued Gromov--Witten invariants of the total space $P$ of the fibration. To start, one needs to choose a symplectic form on $P$. As $P \to B$ is a Hamiltonian fibration, there exists a Hamiltonian connection and hence a coupling form (see \cite[Section 6]{McDuff_Salamon_1998}) $\Omega_P^0 \in \Omega^2(P)$ which is a closed extension of the fiberwise symplectic form. Let $\omega_B$ be the symplectic form on $B$. For sufficiently large $\kappa>0$, the 2-form  
\beqn
\Omega_P^\kappa:= \Omega_P^0 + \kappa \pi^* \omega_B \in \Omega^2(P)
\eeqn
is a symplectic form whose deformation class is independent of $\kappa$ and the Hamiltonian connection.

Let $\GW_{0, k}^{P, {\mb k}}$ denote the ${\mb k}$-valued $k$-pointed genus zero GW invariants of the space $(P, \Omega_P^\kappa)$. More precisely, for $\alpha_1, \ldots, \alpha_k \in H^*(P; {\mb k})$, 
\beqn
\GW^{P, {\mb k}}_{0, k}(\alpha_1, \ldots, \alpha_k)  = \sum_{\tilde A \in \pi_2(P)} q^{\Omega_P^\kappa(\tilde A)} \GW^{P, {\mb k}; \tilde A}_{0, k}(\alpha_1, \ldots, \alpha_k).
\eeqn
This invariant depends on the real number $\kappa$. However, if we fix a homotopy class 
\beqn
A \in \pi_2(B)
\eeqn
of the base, then we can consider the contribution 
\beqn
\sum_{\pi_*(\tilde A) = A} q^{\Omega_P^\kappa(\tilde A)} \GW_{0, k}^{P, {\mb k}; \tilde A}(\alpha_1, \ldots, \alpha_k) = q^{\kappa \Omega_B(A)} \sum_{\pi_*(\tilde A) = A} q^{\Omega_P^0(A)} \GW_{0, k}^{P, {\mb k}; \tilde A}(\alpha_1, \ldots, \alpha_k). 
\eeqn
It is $q^{\kappa \Omega_B(A)}$ times of a $\kappa$-independent term.

\begin{thm}\label{thm_nonvanishing}
Under the same setting as in Theorem \ref{thm41}, suppose ${\rm char}({\mb k}) = p$ and 
\beqn
m = \GW_{0, k+2}^{B, A} ( [{\rm pt}], [{\rm pt}], \alpha_1, \ldots, \alpha_k) \notin p {\mb Z}.
\eeqn
Let 
\beqn
(\pi_P^* \alpha_i)_{\mb k} \in H^*(P; {\mb k})
\eeqn
is the image of $\pi_P^* \alpha_i$ under the natural map 
\beqn
H^*(P; {\mb Z}) \to H^*(P; {\mb k}).
\eeqn
Then for any $\beta_0 \in H_*(M; {\mb k} )$, there exists $\beta_\infty \in H_*(M; {\mb k})$ such that 
\beq\label{eqn:GW-nonzero}
\sum_{\pi_*(\tilde A) = A} q^{\Omega_P^0(\tilde A)} \GW_{0, k+2}^{P, {\mb k}; \tilde A}  \Big( {\rm PD}_P (\iota_*(\beta_0)), {\rm PD}_P (\iota_*(\beta_\infty)), (\pi^* \alpha_1)_{\mb k}, \ldots, (\pi^* \alpha_k)_{\mb k} \Big) \neq 0.
\eeq
\end{thm}

The finite characteristic case of Theorem \ref{thm41} follows immediately; the characteristic zero case follows by considering a prime $p$ such that the homology groups do not have $p$-torsions.

\subsection{Proof of Theorem \ref{thm_nonvanishing}}

To prove Theorem \ref{thm_nonvanishing}, we translate it into a more specific statement relating the GW invariants of $P$ with the graph GW invariants of certain Hamiltonian fibrations over $S^2$. 
By the assumption of Theorem \ref{thm41}, one can assume that the intersection \eqref{eqn40} consists of $\tilde m$ points $q_1, \ldots, q_{\tilde m} \in {\mc M}{}_{0, k+2}^{\rm reg}(J_B, A)$ represented by marked smooth spheres
\beqn
(u_l, {\bf z}_l = (z_{l, 0}, z_{l, 1}, \ldots, z_{l, k}, z_{l, \infty})),\ l = 1, \ldots, \tilde m
\eeqn
with signs ${\rm sign}(u_l,{\bf z}_l) \in \{\pm 1\}$. The total sum of these signed counts is the Gromov--Witten invariant $\GW_{0, k+2}^{B, A}([{\rm pt}], [{\rm pt}], a_1, \ldots, a_k)$. Each $u_l$ pulls back a Hamiltonian fibration 
\beqn
u_l^* P \to S^2,\ l = 1, \ldots, \tilde m.
\eeqn
Then there are natural maps
\beqn
\pi_2^{\rm graph}(u_l^* P) \to (\pi_{P/B})_*^{-1}(A) \subset \pi_2(P)
\eeqn
(which may not be injective). Now consider arbitrary $\beta_0, \beta_\infty \in H_*(M; {\mb k})$. Denote
\beqn
\beta_0^P = \iota_* (\beta_0), \beta_\infty^P  = \iota_*(\beta_\infty) \in H_*(P; {\mb k})
\eeqn
be the pushforward induced by the inclusion of a fiber.

\begin{prop}\label{prop44}
One has
\begin{multline}\label{eqn41}
\sum_{\pi_*(\tilde A) = A} q^{\Omega_P^\kappa (\tilde A)} \GW_{0, k+2}^{P, {\mb k}; \tilde A} \Big( {\rm PD}_P( \beta_0^P), {\rm PD}_P(\beta_\infty^P), (\pi_{P/B}^* \alpha_1)_{\mb k}, \ldots, (\pi_{P/B}^* \alpha_k)_{\mb k} \Big) \\
= q^{\kappa \Omega_B(A)} \sum_{l=1}^{\tilde m} {\rm sign}(u_l, {\bf z}_l) \sum_{\tilde A' \in \pi_2^{\rm graph} (u_l^*P)} q^{\tilde \Omega_{u_l^*P }(\tilde A')} \GGW_{0, 2, {\bf z}_l}^{u_l^* P, \tilde A'} (\beta_0, \beta_\infty) \in \Lambda.
\end{multline}
\end{prop}

As all $u_l$ are in the same homotopy class, the Hamiltonian fibrations $u_l^* P \to S^2$ can be identified with $\tilde M_\phi\to S^2$ for a fixed $\phi: S^1 \to {\rm Ham}(M, \omega_M)$ and all coupling forms $\tilde \Omega_{u_l^* P}$ are in the same homology class of a fixed coupling form $\tilde \Omega_\phi \in \Omega^2(\tilde M_\phi)$. Therefore the right hand side of \eqref{eqn41} is equal to 
\beqn
m q^{\kappa \Omega_B(A)} \sum_{\tilde A' \in \pi_2^{\rm graph}(\tilde M_\phi)} q^{\tilde\Omega_\phi(\tilde A' )} \wt\GW{}_{0, 2, {\bf z}}^{\tilde M_\phi, \tilde A' }(\beta_0, \beta_\infty)
\eeqn
for some fixed marking ${\bf z}$. As $m$ is not divisible by $p$, hence invertible in ${\mb k}$, Theorem \ref{thm_nonvanishing} now follows from Proposition \ref{prop44} and Corollary \ref{cor_nondegenerate}.  \qed

\subsection{Proof of Proposition \ref{prop44}}

The proof of Proposition \ref{prop44} is carried out by carefully constructing a global Kuranishi chart on the moduli space of stable maps into the total space and choosing a good FOP transverse perturbation. 

Let $J_B$ be the almost complex structure given in the assumption of Theorem \ref{thm41}. We first choose an almost complex structure $J_P$ on the total space $P$ which is tamed by $\Omega_P^\kappa$ and such that the projection $(P, J_P) \to (B, J_B)$ is pseudo-holomorphic. For each class $\tilde A \in \pi_2(P)$ such that $\pi_*(\tilde A) = A$ consider the moduli space of stable maps 
\beqn
\ov{\mc M}{}_{0, k+2}(J_P, \tilde A).
\eeqn
Then there is a natural projection map 
\beqn
\pi_{P/B}: \ov{\mc M}{}_{0, k+2}(J_P, \tilde A) \to \ov{\mc M}{}_{0, k+2}(J_B, A).
\eeqn
Denote
\beqn
\ov{\mc M}{}_{0, k+2}^{\rm sing}(J_B, A):= \ov{\mc M}{}_{0,k+2}(J_B, A) \setminus {\mc M}{}_{0, k+2}^{\rm reg}(J_B, A)
\eeqn
where we switch the notation ${\mc M}{}_{0, k+2}^{\rm reg}(J_B, A) = {\mc M}{}_{0, k+2}^{\rm simple}(J_B, A)$, and 
\begin{align*}
&\ \ov{\mc M}{}_{0, k+2}^{\rm reg} (J_P, \tilde A) = (\pi_{P/B})^{-1} ({\mc M}{}_{0, k+2}^{\rm reg}(J_B, A)),\ &\ \ov{\mc M}{}_{0, k+2}^{\rm sing}(J_P, \tilde A) = (\pi_{P/B})^{-1} ( \ov{\mc M}{}_{0, k+2}^{\rm sing}(J_B, A) ).
\end{align*}
For the $\tilde m$ specific points $q_l = [u_l, {\bf z}_l] \in {\mc M}{}_{0, k+2}^{\rm reg}(J_B, A)$, their preimages
\beqn
\pi_{P/B}^{-1}(q_l) \subset \ov{\mc M}{}_{0, k+2}(J_P, A)
\eeqn
is naturally identified with the (disjoint) union of 
\beqn
\ov{\mc M}{}_{0, 2, {\bf z}_l}^{\rm graph}( J_P|_{u_l^* P}, \tilde A')
\eeqn
with $\tilde A' \in \pi_2^{\rm graph}(u_l^* P)$ sent to $\tilde A \in \pi_2(P)$. By abuse of notation, denote by 
\beqn
\ov{\mc M}{}_{0, 2, {\bf z}_l}^{\rm graph}(J_P|_{u_l^* P}, \tilde A)
\eeqn
this disjoint union.

\begin{prop}\label{prop_chart_detail}
There exist the following objects.
\begin{enumerate}
\item A smooth almost complex global Kuranishi chart
\beqn
K = (G, V, E, S, \Psi)
\eeqn
for $\ov{\mc M}{}_{0, k+2}(J_P, \tilde A)$. Denote the induced derived orbifold chart by 
\beqn
{\mc C} = ({\mc U}, {\mc E}, {\mc S}, \psi),
\eeqn
then it is equivalent to (see Definition \ref{defn_chart_equivalence}) the derived orbifold chart induced from an AMS global Kuranishi chart.

\item A singular global Kuranishi chart
\beqn
K_{\rm vert} = (G, V_{\rm vert}, E_{\rm vert}, S_{\rm vert}, \Psi_{\rm vert})
\eeqn
which restricts to an almost complex smooth global Kuranishi chart for the open subset $\ov{\mc M}{}_{0,k+2}^{\rm reg}(J_P, A) \subset \ov{\mc M}{}_{0, k+2}(J_P, A)$ 
with a $G$-equivariant commutative diagram
\beq\label{commutative}
\vcenter{ \xymatrix{      E_{\rm vert} \ar[rr]^{\wh\iota} \ar[d]  & &   E \ar[d]\\
                 V_{\rm vert} \ar[rr]_{\iota}  \ar@/^2.0pc/@[][u]^{S_{\rm vert}} & &  V \ar@/_2.0pc/@[][u]_{S}  } }
\eeq
where $\iota$ is a $G$-equivariant continuous embedding covered by the bundle homomorphism $\wh\iota$, which restricts to a $G$-equivariant embedding over $V_{\rm vert}^{\rm reg}$ compatible with the almost complex structures.

\item A $G$-equivariant complex vector bundle $\pi_{\rm hor}: W_{\rm hor} \to V_{\rm vert}^{\rm reg}$, a $G$-invariant neighborhood $W_{\rm hor}^{\rm reg} \subset W_{\rm hor}$ of the zero section, and a $G$-equivariant commutative diagram
\beq\label{eqn:commutative-2}
\xymatrix{ \pi_{\rm hor}^* E_{\rm vert} \oplus \pi_{\rm hor}^* W_{\rm hor} \ar[rr]^-{\wh\theta_{\rm hor}} \ar[d]  &  & E \ar[d] \\
W_{\rm hor}^{\rm reg} \ar[rr]_-{\theta_{\rm hor}}   &  & V }
\eeq
where $\theta_{\rm hor}$ is a homeomorphism onto an open subset extending the embedding of $V_{\rm vert}^{\rm reg}$ into $V$ and $\wh\theta_{\rm hor}$ is a bundle isomorphism to $\theta_{\rm hor}^* E$. Moreover, if we write the $W_{\rm hor}$-component of $S$ restricted along the image of $\theta_{\rm hor}$ as $S_{\rm hor}$ and write the tautological section of $\pi_{\rm hor}^* W_{\rm hor} \to W_{\rm hor}^{\rm reg}$ as $\tau_{W_{\rm hor}}$,  then 
\beqn
\wh\theta_{\rm hor} \circ \tau_{W_{\rm hor}}  =  S_{\rm hor} \circ \theta_{\rm hor}.
\eeqn

\item A $G$-invariant smooth submersive map
\beq\label{projection1}
\tilde \pi_{\rm vert}: V_{\rm vert}^{\rm reg} \to {\mc M}{}_{0,k+2}^{\rm reg}(J_B, A)
\eeq
which makes the following diagram commute
\beqn
\xymatrix{   {\mc S}_{\rm vert}^{-1}(0)\cap V_{\rm vert}^{\rm reg}  \ar[r] \ar[d]   &  {\mc S}_{\rm vert}^{-1}(0)/G   \ar[d] \\
     V_{\rm vert}^{\rm reg} \ar[d]   &  \ov{\mc M}{}_{0, k+2}(J_P, \tilde A)  \ar[d] \\
                           {\mc M}{}_{0, k+2}^{\rm reg}(J_B, A)  \ar[r]   &  \ov{\mc M}{}_{0, k+2}(J_B, A). } 
\eeqn
Furthermore, for each point $q_l = [u_l, {\bf z}_l] \in {\mc M}{}_{0, k+2}^{\rm reg}(J_B, A)$, denote the fiber of $V^{\rm reg}_{\rm vert}$ by $V_{{\rm vert}, q_l}^{\rm reg}$, then the restriction of $K_{\rm vert}$ to $V_{{\rm vert}, q_l}^{\rm reg}$ is equivalent to an AMS global Kuranishi chart on the graph moduli space $\ov{\mc M}{}_{0, 2, {\bf z}_l}^{\rm graph}(\tilde J_P|_{u_l^* P}, \tilde A)$ constructed in Section \ref{section3}.
\end{enumerate}
\end{prop}

The proof of Proposition \ref{prop_chart_detail} is deferred to the next subsection.

Now we can prove Proposition \ref{prop44}. We first choose pseudocycle representatives of the constraints. For each pseudocycle $f_i: W_i \to B$, denote $\tilde W_i $ the total space of the pullback of $P \to B$ and 
\beqn
\tilde f_i: \tilde W_i \to P
\eeqn
the induced map. It is easy to see that $\tilde f_i$ is again a pseudocycle and represents the class ${\rm PD}_P ((\pi_P^* a_i)_{\mb k})$. On the other hand, choose $p$-pseudocycles
\beqn
f_0: W_0 \to M,\ {\rm resp.}\  f_\infty: W_\infty \to M
\eeqn
representing the classes $\beta_0$ resp. $\beta_\infty$. The composition with the inclusions $\iota_0: M \to P|_{p_0}$ resp. $\iota_\infty: M \to P|_{p_\infty}$ we obtain $p$-pseudocycles
\beqn
\tilde f_0 = \iota_0 \circ f_0: W_0 \to P\ {\rm resp.}\ \tilde f_\infty = \iota_\infty \circ f_\infty: W_\infty \to P.
\eeqn

Let the induced derived orbifold chart of $K$ be
\beqn
{\mc C} = ({\mc U}, {\mc E}, {\mc S}, \psi): = (V/G, E/G, S/G, \Psi/G)
\eeqn
and also denote
\beqn
{\mc C}_{\rm vert}^{\rm reg} = ({\mc U}_{\rm reg}, {\mc E}_{\rm vert}, {\mc S}_{\vert}, \psi_{\rm vert}): = (V_{\rm vert}^{\rm reg}/G, E_{\rm vert}/G, S_{\rm vert}/G, \Psi_{\rm vert}/G).
\eeqn
Consider the $\tilde m$ specific points $q_l = [(u_l, {\bf z}_l)] \in U\subset {\mc M}{}_{0, k+2}^{\rm simple}(J_B, A)$ contributing to the Gromov--Witten invariant of $B$ and the fibers
\beqn
{\mc U}^{\rm reg}_{{\rm vert}, q_l}:= \tilde \pi_{\rm vert}^{-1}(q_l) \subset {\mc U}_{\rm vert}^{\rm reg}.
\eeqn
Then the quadruple
\beqn
{\mc C}_{{\rm vert}, q_l} := ({\mc U}^{\rm reg}_{{\rm vert}, q_l}, {\mc E}_{\rm vert}, {\mc S}_{\rm vert})
\eeqn
is a stably complex derived orbifold chart for the graph moduli $\ov{\mc M}{}_{0, 2, {\bf z}_l}^{\rm graph}( \tilde J_P|_{u_l^* P}, \tilde A)$, where we abuse the notations by writing ${\mc E}_{\rm vert} = (E_{\rm vert}/G )|_{{\mc U}^{\rm reg}_{\rm vert}(q_l)}$ and ${\mc S}_{\rm vert} = (S_{\rm vert}/G) |_{{\mc U}^{\rm reg}_{\rm vert}(q_l)}$. 
Choose a distance function on ${\mc U}$. For any $\delta>0$, let 
\beqn
{\mc U}_{\rm vert}^{{\rm reg},\delta}(q_l) \subset {\mc U}_{{\rm vert}}^{\rm reg}
\eeqn
be the $\delta$-neighborhood of ${\mc U}_{{\rm vert}, q_l}^{\rm reg}$ inside ${\mc U}_{\rm vert}^{\rm reg}$.

Now, fix FOP transverse perturbations for each $q_l$
\beqn
{\mc S}_{{\rm vert}, q_l}': {\mc U}_{{\rm vert}, q_l}^{\rm reg} \to {\mc E}_{\rm vert}.
\eeqn

\noindent {\bf Claim.} One can extend ${\mc S}_{{\rm vert}, q_l}'$ to a FOP transverse perturbation 
\beqn
{\mc S}_{\rm vert}': {\mc U}_{\rm vert}^{\rm reg} \to {\mc E}_{\rm vert}.
\eeqn

\vspace{0.1cm}
\noindent {\it Proof of the claim.} Notice that ${\mc M}_{0, k+2}^{\rm reg}(J_B, A) \subset \ov{\mc M}_{0, k+2}(J_B, A)$ is a manifold rather than an orbifold, the normal bundle of ${\mc U}_{\rm vert}^{\rm reg}(q_l)$ inside ${\mc U}_{\rm vert}^{\rm reg}$, which is the pullback of the tangent space $T_{q_l} {\mc M}_{0, k+2}^{\rm reg}(J_B, A)$ under the projection map \eqref{projection1}, is a vector bundle whose fibers are trivial representations of the isotropy groups. Then by (6) of Theorem \ref{thm_FOP_property}, one can extend the FOP transverse perturbation to a FOP transverse perturbation
\beqn
{\mc S}_{\rm vert}': {\mc U}_{\rm vert}^{\rm reg} \to {\mc E}_{\rm vert}.
\eeqn
This finishes the proof of this claim.
\vspace{0.1cm}

Next, look at the ``further thickened" chart ${\mc U}:=V/G$ which contains ${\mc U}_{\rm vert}$ by the diagram \eqref{commutative}. As $W_{\rm hor} \to V_{\rm vert}^{\rm reg}$ is a $G$-equivariant vector bundle, it descends to an orbifold vector bundle 
\beqn
{\mc W}_{\rm hor} \to {\mc U}_{\rm vert}^{\rm reg}.
\eeqn
By (3) of Proposition \ref{prop_chart_detail}, inside ${\mc U}$ there is an open subset identified with ${\mc W}_{\rm hor}^{\rm reg} \subset {\mc W}_{\rm hor}$ which contains ${\mc U}_{\rm vert}^{\rm reg}$. By choosing an appropriate metric on the bundle ${\mc W}_{\rm hor}$, we may assume that ${\mc W}_{\rm hor}^{\rm reg}$ is the open disk bundle of radius $2\epsilon$ for some $\epsilon > 0$, denoted by ${\mc W}_{\rm hor}^{2\epsilon}$. 

\vspace{0.1cm}

\noindent {\bf Claim.} There exists an FOP transverse extension of ${\mc S}_{\rm vert}'$ to the open subset ${\mc W}_{\rm hor}^{2\epsilon}$, denoted by ${\mc S}'$, satisfying the following conditions.
\begin{enumerate}
\item Over ${\mc W}_{\rm hor}^\epsilon$, the section ${\mc S}'$ is given by the stabilization of ${\mc S}_{\rm vert}'$ by the tautological section.

\item $({\mc S}')^{-1}(0) = ({\mc S}_{\rm vert}')^{-1}(0)$.
\end{enumerate}

\noindent {\it Proof of the claim.}
We introduce some notations for the argument. Let $x$ be the coordinate on ${\mc U}_{\rm vert}^{\rm reg}$ and $y$ be the fiber coordinate on ${\mc W}_{\rm hor}$. We can write any extension ${\mc S}'$ as
\beqn
{\mc S}'(x, y) = \Big( {\mc S}_{\rm vert}'(x,y), {\mc S}_{\rm hor}'(x,y) \Big)
\eeqn
where the first component takes value in $\pi_{\rm hor}^* {\mc E}_{\rm vert}$ and the second one takes value in $\pi_{\rm hor}^* {\mc W}_{\rm hor}$. The value of ${\mc S}_{\rm vert}'(x,0)$ has been fixed. Withint ${\mc W}_{\rm hor}^\epsilon$, we define 
\beq
{\mc S}'(x,y) = ({\mc S}_{\rm vert}'(x,0), y)
\eeq
where $y$ is regarded as the tautological section of $\pi_{\rm hor}^* {\mc W}_{\rm hor}$. Then by (5) of Theorem \ref{thm_FOP_property} (stabilization property), ${\mc S}'$ is FOP transverse and obviously $({\mc S}')^{-1}(0) = ({\mc S}'_{\operatorname{vert}})^{-1}(0)$. To extend to ${\mc W}_{\rm hor}^{2\epsilon}$, we keep ${\mc S}_{\rm hor}'$ to be the tautological section and extend ${\mc S}_{\rm vert}'$ arbitrarily. This will not create any more zeroes and hence still FOP transverse. \hfill {\it End of the Proof of the claim.}

Lastly, by the the ``CUDV property'' of Theorem \ref{thm_FOP_property}, one can find an FOP transverse perturbation on ${\mc U}$ which agrees with ${\mc S}'$ over $\theta_{\rm hor}(\ov{W_{\rm hor}^\epsilon})$ constructed above. By abuse of notation, still denote the perturbation by ${\mc S}'$. Notice that by (2) of Theorem \ref{thm_FOP_property}, one can assume
\beqn
\| {\mc S} - {\mc S}' \|_{C^0} \leq \delta
\eeqn
for any $\delta>0$. 

Now we look at the intersection numbers. Consider 
\beqn
\Big( ( {\mc S}')^{-1}(0) \cap {\mc U}_{\rm free} \Big) \cap \ev_P^{-1} \Big( f_0(W_0) \times f_\infty(W_\infty) \times \prod_{i=1}^k \tilde f_i (\tilde W_i) \Big)
\eeqn
whose signed count coincides with the GW invariant
\beqn
\GW_{0, k+2}^{P, {\mb k}; \tilde A} \big( {\rm PD}_P(\beta_0^P), {\rm PD}_P(\beta_\infty^P), (\pi_{P/B}^* \alpha_1)_{\mb k}, \ldots, (\pi_{P/B}^* \alpha_k)_{\mb k} \big).
\eeqn
Also, obviously
\begin{multline*}
\Big( ( {\mc S}')^{-1}(0) \cap {\mc U}_{\rm free} \Big) \cap \ev_P^{-1} \Big( \tilde f_0( \tilde W_0) \times \tilde f_\infty( \tilde W_\infty) \times \prod_{i=1}^k \tilde f_i (\tilde W_i) \Big) \\
\subset \Big\{ x \in {\mc U}\ |\ \pi_{P/B}( \ev_0(x)) = p_0, \pi_{P/B}(\ev_\infty(x)) = p_\infty,\ \pi_{P/B} \circ \ev_i( x) \in f_i(W_i) \subset B \Big\}
\end{multline*}

As inside ${\mc U}$, the perturbed zero locus $({\mc S}')^{-1}(0)$ can be arbitrarily close to ${\mc S}^{-1}(0)$, the above intersection points can be arbitrarily close to points in ${\mc S}^{-1}(0) \cong \ov{\mc M}{}_{0, k+2}(J_P, A)$ which project to the points $q_l = [u_l, {\bf z}_l]$. In particular, the above intersection points are in the open subset
\beqn
\bigcup_{l=1}^{\tilde m} {\mc W}_{\rm hor}^\epsilon|_{{\mc U}_{\rm vert}^{{\rm reg}, \delta}(q_l)} \subset {\mc W}_{\rm hor}^\epsilon.
\eeqn
By our construction of ${\mc S}'$, within the disk bundle ${\mc W}_{\rm hor}^\epsilon$ the perturbation is a stabilization of ${\mc S}_{\rm vert}'$, hence 
\begin{multline*}
\Big( ( {\mc S}')^{-1}(0) \cap {\mc U}_{\rm free} \Big) \cap \ev_P^{-1} \Big( \tilde f_0 ( \tilde W_0) \times \tilde f_\infty (\tilde W_\infty) \times \prod_{i=1}^k \tilde f_i (\tilde W_i) \Big) \\
= \bigsqcup_{l = 1}^{\tilde m} \Big( ( {\mc S}_{\rm vert}')^{-1}(0)_{\rm free} \cap {\mc U}_{{\rm vert}, q_l}^{\rm reg}  \Big) \cap \ev_0^{-1} ( \tilde f_0 ( \tilde W_0)) \cap \ev_\infty^{-1}( \tilde f_\infty (\tilde W_\infty)).
\end{multline*}


As the fiber ${\mc U}_{\rm vert}^{\rm reg}(q_l)$ arises from a derived orbifold chart of the moduli space $\ov{\mc M}{}_{0, 2, {\bf z}_l}^{\rm graph}( J_P|_{u_l^* P}, \tilde A)$ and the restriction of ${\mc S}_{\rm vert}'$ to ${\mc U}_{\rm vert}^{\rm reg}(q_l)$ is an FOP transverse perturbation, the signed count of the above intersection set is (up to signs) equal to the graph GW invariant
\beqn
\wt\GW{}_{0, 2, {\bf z}_l}^{u_l^* P, \tilde A}(\beta_0, \beta_\infty) \in {\mb k}.
\eeqn
The detailed sign verification is given in Subsection \ref{subsection_signs}. Summing over all the points in $\Big( ( {\mc S}')^{-1}(0) \cap {\mc U}_{\rm free} \Big) \cap \ev_P^{-1} \Big( \tilde f_0 ( \tilde W_0) \times \tilde f_\infty ( \tilde W_\infty) \times \prod_{i=1}^k \tilde f_i (\tilde W_i) \Big) $, we obtain the desired statement. \qed



\subsection{Proof of Proposition \ref{prop_chart_detail}} 

\subsubsection{The AMS construction}

As before, the first step of construction a global Kuranishi chart is to approximate the symplectic form $\Omega_P^\kappa$ by a rational one, denoted by $\Omega_P'\in \Omega^2(P)$. When $\Omega_P'$ is sufficiently close to $\Omega_P^\kappa$, $J_P$ is still tamed by $\Omega_P'$. After rescaling, one may assume that $\Omega_P'$ is integral. For the curve class $\tilde A \in \pi_2(P)$, denote 
\beqn
d(\tilde A) = \langle \Omega_P', \tilde A \rangle \in {\mb Z}.
\eeqn

Second, for each $d > 0$, consider the moduli space $\ov{\mc M}_{0, k+2}(\mb{CP}^d,d)$ (of unparametrized stable maps). The group $G_d = U(d+1)$ acts on $\ov{\mc M}_{0, k+2}(\mb{CP}^d, d)$. There is a $G_d$-invariant subset $B_{k+2, d} \subset \ov{\mc M}_{0, k+2}(\mb{CP}^d, d)$ consisting of configurations whose images are not contained in any hyperplane. $B_{k+2, d}$ is a smooth complex $G_d$-manifold. Then consider the universal curve 
\beqn
C_{k+2, d} \to B_{k+2, d}.
\eeqn
Let $C_{k+2, d}^*\subset C_{k+2, d}$ be the subset of points which are not marked points or nodal points of fibers. 

Similar to the AMS construction for the graph moduli, one consider the following objects.
\begin{defn}
\begin{enumerate}

\item A finite-dimensional approximation scheme on $C_{k+2, d}$ is a representation $W$ of $G_d$ together with an equivariant linear map 
\beqn
\iota: W \to C^\infty_c ( C_{k+2, d}^* \times P, \Omega^{0,1}_{C_{k+2, d}^*/ B_{k+2, d}} \otimes TP).
\eeqn

\item Given $\tilde A \in \pi_*^{-1}(A)$ and a finite-dimensional approximation scheme $(W, \iota)$ over $C_{k+2, d}$ for $d = d(\tilde A)$, the {\bf pre-thickening} is the set
\beq\label{eqn:pre-thicken}
V^{\rm pre} = \Big\{ (\rho, u, e)\ |\ \rho \in B_{k+2, d},\ u: C_\rho \to P,\ e \in W,\ \ov\partial_{J_P} u + \iota(e)|_{C_\phi} = 0,\ [u] = \tilde A \Big\}.
\eeq
\end{enumerate}
\end{defn}

\begin{defn}\label{defn_partial_regularity}
We say that a finite-dimensional approximation scheme $(W, \iota)$ is regular over an open subset $U \subset \ov{\mc M}_{0,k+2}(J_P, \tilde A)$ if for each $y \in U$ and each $(\rho, u, 0) \in V^{\rm pre}$ sent to $y$ by the forgetful map, the linearization of the equation $\ov\partial_{J_P} u + \iota(e) = 0$ is surjective.
\end{defn}

Recall that we have a natural projection $\pi_{P/B}: \ov{\mc M}{}_{0, k+2}(J_P, \tilde A) \to \ov{\mc M}{}_{0, k+2}(J_B, A)$.

\begin{lemma}\label{lemma_vertical_obstruction}
One can choose a finite-dimensional approximation scheme $(W^{\rm vert}, \iota^{\rm vert})$ (called the vertical obstruction space) satisfying the following two conditions.
\begin{enumerate}
    \item For all $e \in W^{\rm vert}$, $\iota^{\rm vert} (e)$ lies in $T^{\rm vert} P$.

    \item For any compact subset $K \subset \ov{\mc M}{}_{0, k+2}^{\rm reg}(J_B, A)$, $(W^{\rm vert}, \iota^{\rm vert})$ is regular near $\pi_{P/B}^{-1}(K)$.
\end{enumerate}
\end{lemma}

\begin{proof}
This is because downstairs away from multiple covers and nodal curves the moduli space is already regular. Therefore obstructions only lie in the vertical direction, so we can choose a sufficiently large finite dimensional approximation scheme lying in the vertical direction of the tangent bundle of $P$.
\end{proof}

We first use the vertical obstruction space to construct a Kuranishi chart which is a manifold only away from certain bad locus. We may shrink the open subset ${\mc M}{}_{0, k+2}^{\rm reg}(J_B, A)$ given by the assumption of Theorem \ref{thm41} to a precompact neighborhood of the points $p_1, \ldots, p_{\tilde m}$. Then we may assume that there is a finite dimensional approximation scheme $(W_{\rm vert}, \iota_{\rm vert})$ as in Lemma \ref{lemma_vertical_obstruction}, such that it is regular over the corresponding open subset of $\ov{\mc M}{}_{0, k+2}^{\rm reg}(J_P, \tilde A)$.  Let the pre-thickening obtained from such data be $V_{\rm vert}^{\rm pre}$. 

To proceed, we use framings to construct the actual AMS chart. For each triple $(\rho, u, e)$, the pullback form $u^* \Omega_P' \in \Omega^2(C_\rho)$ has integral being $d = d(\tilde A)$, hence there is a Hermitian holomorphic line bundle $L_u \to C_\rho$, unique up to unitary isomorphism, whose curvature form is $-2 \pi {\bf i} u^* \Omega_P'$. A framing on $(\rho, u, e)$ is a basis of $H^0(L_u)$
\beqn
F = (f_0, \ldots, f_d).
\eeqn
Then define 
\beqn
V_{\rm vert}:= \Big\{(\rho, u, e, F)\ |\ (\rho, u, e) \in V_{\rm vert}^{\rm pre},\ F\ {\rm is\ a\ framing}\Big\}
\eeqn
which has the structure of a $G_d$-equivariant principal $G_d^{\mb C}$-bundle. Define the obstruction bundle $E_{\rm vert} \to V$ as
\beqn
E_{\rm vert} = W_{\rm vert} \oplus \pi_{V/B}^* TB_{k+2, d} \oplus {\mf g}_d.
\eeqn
Choosing a $G_d$-invariant Riemannian metric on $B_{k+2, d}$ and denote the exponential map by $\exp_B$, which identifies a neighborhood $\Delta^+(B_{k+2, d})$ of the diagonal in $B_{k+2, d}\times B_{k+2, d}$ with a neighborhood of the zero section of the tangent bundle of $B_{k+2, d}$. Shrink $V$ to the open subset where $(\rho, \rho_F) \in \Delta^+(B_{k+2, d})$ and where the Hermitian matrix $H_F$ with entries
\beqn
\int_{C_\rho} \langle f_i, f_j \rangle u^* \Omega_P'
\eeqn
is positive definite. Define the Kuranishi section $S_{\rm vert}: V_{\rm vert} \to E_{\rm vert}$ as 
\beq\label{eqn:kuranishi-vert}
S_{\rm vert} (\rho, u, e, F) = \Big( e, \exp_B^{-1}(\rho, \rho_F), \exp_H^{-1}(H_F) \Big).
\eeq
If $S_{\rm vert} (\rho, u, e, F) = 0$, we see $u: C_\rho \to P$ is a genuine $J_P$-holomorphic map. Then define the $G_d$-equivariant map
\beqn
\psi_{\rm vert}: S_{\rm vert}^{-1}(0)/G_d \to \ov{\mc M}{}_{0, k+2}(J_P,\tilde A),\ \tilde \psi( \rho, u, 0, F)  = [u: C_\rho \to P].
\eeqn
Similar to Lemma \ref{lemma34}, one has 

\begin{lemma}
$\psi_{\rm vert}$ is a homeomorphism. \qed
\end{lemma}

Therefore one obtains a singular global Kuranishi chart  
\beqn\label{eqn:global-vertical}
K_{\rm vert} = (G_d, V_{\rm vert}, E_{\rm vert}, S_{\rm vert}, \psi_{\rm vert})
\eeqn
which is regular over ${\mc M}{}_{0, k+2}^{\rm reg}(J_P, \tilde A)$.

One can then choose another finite-dimensional approximation scheme $(W_{\rm hor}, \iota_{\rm hor})$ over $C_{k+2, d}$ (called the horizontal obstruction space\footnote{Although the corresponding sections do not necessarily take value in the horizontal part of the pulled-back tangent bundle, we hope that this notion will not cause any confusion.}) and denote
\beqn
(W, \iota) = (W_{\rm vert} \oplus W_{\rm hor}, \iota_{\rm vert} \oplus \iota_{\rm hor})
\eeqn
so that it is regular everywhere in the sense of Definition \ref{defn_partial_regularity}. Notice that $W_{\rm hor}$ defines a $G_d$-equivariant vector bundle 
\beqn
\pi_{\rm hor}: W_{\rm hor} \to V_{\rm vert}
\eeqn 

Parallel to the previous discussions on the singular global Kuranishi chart, now we can construct a global Kuranishi chart for $\ov{\mc M}{}_{0, k+2}(J_P,\tilde A)$. Using the finite-dimensional approximation scheme $(W, \iota)$, the pre-thickening $V^{\rm pre}$ is defined as in \eqref{eqn:pre-thicken}. We obtain a global Kuranishi chart
\beqn
K = (G_d, V, E, S, \psi)
\eeqn
where $V = \{(\rho, u, e, F)\ |\ (\rho, u, e) \in V^{\rm pre},\ F\ {\rm is\ a\ framing}\}$; $E = W \oplus \pi_{V/B}^* TB_{k+2, d} \oplus {\mf g}_d$, which differs from $E_{\rm vert}$ by the direct summand $W_{\rm hor}$; the Kuranishi section $S$ is defined using the same formula \eqref{eqn:kuranishi-vert}; the map $\psi$ is defined similarly to $\psi_{\rm vert}$, which gives rise to a homeomorphism $S_{\rm vert}^{-1}(0)/G_d \xrightarrow{\sim} \ov{\mc M}{}_{0, k+2}(J_P,\tilde A)$.

\begin{lemma}
Away from the horizontally singular locus, $V$ is a stabilization of $V_{\rm vert}$. More precisely, there is a $G_d$-invariant open neighborhood $V_{\rm vert}^{\rm reg} \subset V_{\rm vert}$ of $\psi_{\rm vert}^{-1} \left( \ov{\mc M}{}_{0, k+2}^{\rm reg}(J_P, \tilde A)\right)$ which is a topological manifold, a $G_d$-invariant open neighborhood of the zero section of $W_{\rm hor}|_{V_{\rm vert}^{\rm reg}}$, denoted by $W_{\rm hor}^{\rm reg}$, a $G_d$-equivariant bundle map $\wh\theta_{\rm hor}$ covering a $G_d$-equivariant map $\theta_{\rm hor}$
\beq\label{eqn_stabilization_map}
\xymatrix{    \pi_{\rm hor}^* E_{\rm vert} \oplus \pi_{\rm hor}^* W_{\rm hor} \ar[rr]^-{\wh\theta_{\rm hor}} \ar[d]   &  &     E \ar[d] \\
               W_{\rm hor}^{\rm reg} \ar[rr]_-{\theta_{\rm hor}}  &  & V }
\eeq
satisfying the following conditions. 
\begin{enumerate}

\item $\theta_{\rm hor}$ is a homeomorphism onto an open subset whose restriction to the zero section coincides with the embedding $V_{\rm vert}^{\rm reg} \hookrightarrow V$.

\item Let $S_{\rm hor}: V \to W_{\rm hor}$ be the $W_{\rm hor}$-component of the Kuranishi map $S: V \to E$ and let $\tau_{W_{\rm hor}}: W_{\rm hor}^{\rm reg} \to \pi_{\rm hor}^* W_{\rm hor}$ be the tautological section. Then 
\beqn
\wh\theta_{\rm hor} \circ \tau_{W_{\rm hor}} = S_{\rm hor} \circ \theta_{\rm hor}.
\eeqn
\end{enumerate}
\end{lemma}

\begin{proof}
Consider the forgetful map $V \to B_{k+2, d}$. The fiber of this forgetful map at $\rho \in B_{k+2, d}$, denoted by $V(\rho)$, is the zero locus of a smooth Fredholm map from the pseudoholomorphic graph equation over the smooth or nodal curve $C_\rho$. Hence $V(\rho)$ is naturally a smooth manifold. There is a subset $V_{\rm vert}(\rho) \subset V(\rho)$ which is a submanifold near the regular locus. The normal bundle can be identified with the bundle $W_{\rm hor}$. 

We can choose a $G_d$-invariant family of fiberwise Riemannian metrics on $V(\rho)$ which varies continuously in $\rho$. This is possible as the vertical tangent bundle $T^{\rm vert} V$ exists. Then the map $\theta_{\rm hor}$ can be constructed using the fiberwise exponential map. Indeed, it follows from the construction and the regularity assumption that for any point in $v \in V_{\rm vert}^{\rm reg} \subset V$ over $\rho$, the vertical tangent space of $V$ at this point admits a decomposition $T_v^{\rm vert} V = T_v V_{\rm vert}^{\rm reg} \oplus W_{\rm hor}$, where $T_v V_{\rm vert}^{\rm reg}$ is well-defined due to the regularity of the singular global Kuranishi chart near $v$. Then the fiberwise normal bundle of $V_{\rm vert}^{\rm reg} \subset V$ is identified with $W_{\rm hor} \to V_{\rm vert}^{\rm reg}$. So, for any tangent vector lying in $W_{\rm hor}$ with sufficiently small norm, we can define $\theta_{\rm hor}$ using the exponential map at $v$. By varying over $\rho$, we see that $\theta_{\rm hor}$ defines a homeomorphism onto an open subset which agrees with the embedding $V_{\rm vert} \hookrightarrow V$ along the zero section. As the diagram \eqref{eqn_stabilization_map} commutes along $V_{\rm vert}^{\rm reg}$, one can use parallel transports to construct the bundle map $\wh\theta_{\rm hor}$.

Now by the implicit function theorem, the restriction of $S_{\rm hor}$ to the fibers of $W_{\rm hor}^{\rm reg}$ is a homeomorphism onto a small disk bundle of $W_{\rm hor} \subset E$. By reparametrizing the map $\theta_{\rm hor}$, and accordingly, $\wh\theta_{\rm hor}$, one can ensure that $S_{\rm hor}$ agrees with $\tau_{W_{\rm hor}}$ after being pulled back by $\theta_{\rm hor}$.
\end{proof}

\subsubsection{Smoothing}\label{subsec:smoothing}

We need to construct smooth structures on the base of the Kuranishi chart $V$ which contains a smooth submanifold $V_{\rm vert}^{\rm reg}$. The basic idea is as follows. First, following the argument of \cite{AMS}, by using the topological submersion $V_{\rm vert}^{\rm reg} \to B_{k+2, d}$, one can construct a stable smoothing on $V_{\rm vert}^{\rm reg}$. Then as $V$ is locally a bundle over $V_{\rm vert}^{\rm reg}$, one obtains a natural smooth structure of the thickening. Finally, we use a relative smoothing procedure to obtain a smooth structure on $V$.

The first stage of this construction is similar to the situation of Lemma \ref{lemma_smoothing}. The forgetful map 
\beqn
V_{\rm vert}^{\rm reg} \to B_{k+2, d}
\eeqn
is a $G_d$-equivariant topological submersion with fiberwise smooth structures varying sufficiently regularly, a $C^1_{loc}$ $G$-bundle. This allows us to construct a vector bundle lift of the tangent microbundle $T_\mu V_{\rm vert}^{\rm reg}$. Lashof's theorem\footnote{Note that Lashof's smoothing theorem only requires finitely many orbit types rather than the compactness of the ambient topological manifold.} implies the existence of a stable $G_d$-smoothing, i.e., a smooth structure on $V_{\rm vert}^{\rm reg}\times R$ for some $G_d$-representation $R$. One can indeed include $R$ into the vertical obstruction space $W_{\rm vert}$ such that $R$ is mapped to zero by $\iota_{\rm hor}$. We can then assume $V_{\rm vert}^{\rm reg}$ is a smooth $G_d$-manifold. 

We need the smoothing be compatible with the smooth structures on the moduli space of holomorphic curves in the base $B$. Notice that since the obstruction space $W_{\rm vert}$ is in the vertical direction of the fibration $P \to B$, there is a natural $G_d$-invariant projection map 
\beq\label{eqn:projection-PB}
V_{\rm vert}^{\rm reg} \to {\mc M}{}_{0, k+2}^{\rm reg} (J_B, A)
\eeq
where the codomain has a natural smooth structure. The compatibility of smooth structures is achieved by \cite[Lemma 4.5]{AMS}, which ensures that the above projection map is a smooth submersion by stabilizing the domain. We keep the same notations for the subsequent discussions.

Next, as the bundle $W_{\rm hor} \to V_{\rm vert}^{\rm reg}$ is trivial, it is automatically a smooth $G_d$-equivariant vector bundle. The homeomorphism $\theta_{\rm hor}$ in \eqref{eqn_stabilization_map} then induces a $G_d$-smooth structure on the image of $\theta_{\rm hor}$, which is an open subset of $V$.

We would like to extend the smoothing onto the whole of $V$ using the relative version of Lashof's smoothing theorem proved in \cite[Appendix B]{Bai_Xu_Arnold}. We need to specify a neighborhood of the singular region where we modify the existing smoothing. The region is selected properly in order to compare the GW invariants of the total space $P$ with the GW invariants of the base $B$.

Choose a $G_d$-invariant distance function on the base $V$ of the Kuranishi chart. Choose a $G_d$-invariant distance function on $V$. For any $\epsilon>0$, define
\beqn
V^{{\rm sing} + \epsilon} \subset V
\eeqn
be the open $\epsilon$-neighborhood of $\Psi^{-1}( \ov{\mc M}{}_{0, k+2}^{\rm sing}(J_P, \tilde A))$. Notice that for any $\epsilon>0$, the union 
\beqn
V^\epsilon:= {\rm Im} (\theta_{\rm hor} ) \cup V^{{\rm sing} +\epsilon}
\eeqn
is a $G_d$-invariant open neighborhood of $S^{-1}(0)$ in $V$. One can fix $\epsilon>0$ such that 
\beqn
\pi_{P/B}^{k+2} \Big( \ev_P \big( V^{{\rm sing}+\epsilon} \big) \Big) \cap  \Big( \{p_0\}\times \{p_\infty\}\times \prod_{i=1}^k f_i (W_i)\Big) = \emptyset. 
\eeqn
Then using the relative version of Lashof's theorem \cite[Theorem B.3]{Bai_Xu_Arnold}, one can construct a stable smoothing on $V^\epsilon$ which coincides with the smooth structure on ${\rm Im}(\theta_{\rm hor})$ taking product with the corresponding $G_d$-representation away from $V^{{\rm sing}+\epsilon}$. By absorbing the additional representation of $G_d$ needed for the stable smoothing into the obstruction space $W_{\rm vert}$, one may assume that the stable smoothing is actually a smoothing.

On the other hand, the obstruction bundle can always be identified with a smooth vector bundle. Hence we obtained a smooth global Kuranishi chart for $\ov{\mc M}{}_{0, k+2}(J_P, \tilde A)$, denoted by 
\beqn
K^\epsilon = (G_d, V^\epsilon, E, S, \Psi).
\eeqn

\subsubsection{Wrapping up the proof}
With the above preparations, we can finish the proof.
\begin{proof}[Proof of Proposition \ref{prop_chart_detail}]
    Fixing an $\epsilon > 0$, we can use the global Kuranish chart $K = K^\epsilon$. As discussed in the proof of Proposition \ref{prop:d-chart}, by choosing all the involved stabilizations to be complex, the fact that the vertical tangent spaces can be identified with the index of Cauchy--Riemann operators ensures that we can build an almost complex $K$. As this chart is built from the finite-dimensional approximation $(W, \iota)$ after stabilization and shrinking, it follows from the construction that $K$ is equivalent to an AMS global Kuranishi chart. This proves (1).

    The singular global Kuranishi chart $K_{\rm vert} = (G_d, V_{\rm vert}, E_{\rm vert}, S_{\rm vert}, \psi_{\rm vert})$ for $\ov{\mc M}{}_{0,k+2}(J_P, A)$ from \eqref{eqn:global-vertical} is regular over the open subset $\ov{\mc M}{}_{0,k+2}^{\rm reg}(J_P, A)$. Then the commutativity of \eqref{commutative} and \eqref{eqn:commutative-2}, and the agreement of $S_{\rm hor}$ and $\tau_{W_{\rm hor}}$, follow from \eqref{eqn_stabilization_map}, and the compatibility of smooth structures, in particular, the assertion that $\iota: V_{\rm vert}^{\rm reg} \to V$ is a smooth embedding, follows from the relative smoothing construction in Section \ref{subsec:smoothing}. Because the map $\iota: V_{\rm vert}^{\rm reg} \to V$ covers the identity map on the complex manifold $B_{k+2, d}$, and for each fiber over $\rho \in B_{k+2, d}$, a homotopy of the linearized Cauchy--Riemann operator of an element $v \in V_{\rm vert}^{\rm reg}$ to a complex linear operator can be extended to such a homotopy viewing $v \in V$, the fact that we can choose $W_{\rm hor}$ to be a complex vector space implies that the embedding is compatible with the almost complex structures. This proves (2) and (3).

    As for (4), the submersive property follows from a technical lemma of Abouzaid--McLean--Smith \cite[Lemma 4.5]{AMS}. Indeed, for the projection \eqref{eqn:projection-PB}, upon choosing a further stabilization of the chart $V_{\rm vert}^{\rm reg}$ by the bundle 
    \beqn
    \pi_{P/B}^* T{\mc M}{}_{0, k+2}^{\rm reg}(J_B, A)
    \eeqn
    which admits a $G_d$-invariant smooth structure, we can find a $G_d$-invariant smooth submersive extension of the map $\pi_{P/B}$ to ${\mc M}{}_{0, k+2}^{\rm reg}(J_B, A)$. The commutativity of the diagram follows from the construction. The differences between the restriction of $K_{\rm vert}$ to $V_{{\rm vert}, q_l}^{\rm reg}$ and the global Kuranishi chart arising from the proof of Proposition \ref{prop:d-chart} are
    \begin{itemize}
    \item the (approximations of) the coupling $2$-forms,
    \item the finite dimensional approximation scheme which regularizes the moduli spaces,
    \item the Riemannian metrics required in the definition of the exponential maps on the base.
    \end{itemize}
    Up to equivalence of global Kuranishi chart, these differences can be removed again from the doubly framed curve construction (cf. \cite[Section 6.10]{AMS} and \cite[Section 4.9]{AMS2}). Because of the isotopic uniqueness of stable smoothing, this finishes the proof.
\end{proof}

\subsection{Sign comparison}\label{subsection_signs}

\subsubsection{Linearized operators}

Here we show that the linearization of the nonlinear Cauchy--Riemann equation into the total space $P$ is ``block upper-triangular.'' Consider the fibration $\pi_P: P \to B$. The Hamiltonian connection we have specified provides a decomposition
\beq\label{tangent_bundle_splitting}
TP \cong T^{\rm vert} P \oplus T^{\rm hor} P
\eeq
where $T^{\rm vert} P \subset TP$ is a canonical subbundle and $T^{\rm hor}P$ is isomorphic to the pullback $\pi_P^* TB$. Then, we can choose the almost complex structure $J_P$ to be the form 
\beqn
J_P = J_P^{\rm vert} \oplus \pi_P^* J_B.
\eeqn
To linearize the corresponding $\ov\partial$-operator, one needs to specify a (complex-linear) connection on $TP$ and a way to identify nearby maps (an exponential map); see details in \cite[Section 3]{McDuff_Salamon_2004}. As the almost complex structure $J_P$ respects the splitting \eqref{tangent_bundle_splitting}, one can choose a complex linear connection $\nabla^{TP}$ of the form 
\beqn
\nabla^{TP}:= \nabla^{T^{\rm vert} P} \oplus \pi_P^* \nabla^{TB}.
\eeqn
On the other hand, choose a Riemannian metric on $P$ of the form 
\beqn
g^{TP} = g^{T^{\rm vert} P} \oplus \pi_P^* g^{TB}.
\eeqn
Given any point $x \in P$ and $\xi \in T_x P$, write $\xi = \xi^{\rm vert} + \xi^{\rm hor}$. Then define
\beqn
\wt\exp_x \xi = \exp_{ \exp_x^{\rm vert} ( \xi^{\rm vert})} ( \hat \xi^{\rm hor})
\eeqn
where $\exp^{\rm vert}$ is the fiberwise exponential map with respect to the fiberwise metric $g^{T^{\rm vert} P}$ and the vector $\hat \xi^{\rm hor}$ is the image of $\xi^{\rm hor}$ via the parallel transport along the path $\exp_x^{\rm vert}(t \xi^{\rm vert})$ with respect to the connection $\pi_P^* \nabla^{TB}$. Obviously, $\wt\exp_x$ is a local diffeomorphism. One can see that both the connection and the exponential map preserve fibers. The connection and the exponential map induce the linearization of the Cauchy--Riemann operator. One can see the following lemma is true because the linearized Cauchy--Riemann operator is compatible with the connection.

\begin{lemma}\label{lemma510}
Let $u: \Sigma \to P$ be a smooth map. Then with respect to the horizontal-vertical decomposition \eqref{tangent_bundle_splitting}, the linearized operator $D_u$ is of the form
\beq\label{upper-triangular}
D_u = \left[ \begin{array}{cc} D_u^{\rm vert} & * \\ 0 & D_u^{\rm hor} \end{array}\right].
\eeq
\end{lemma}

If we only deform in the fiber direction, then it is easy to see that the $D_u^{\rm vert}$ coincides with the linearization of the equation for pseudoholomorphic graphs. Moreover, we can verify that the horizontal part is the same as the linearized operator on the base.

\begin{lemma}\label{lemma511}
Denote $\uds u:= \pi_P\circ u: \Sigma \to B$ and let $D_{\uds u}: \Gamma( \uds u^* TB) \to \Omega^{0,1}(\uds u^* TB)$ be the linearization corresponding coming from the $J_B$-linear connection $\nabla^{TB}$. Then one has 
\beqn
D_u^{\rm hor} = \pi_P^* D_{\uds u}.
\eeqn
\end{lemma}

\begin{proof}
Let $\xi^{\rm hor}$ be a vector field along $u$ taking values in $T^{\rm hor} P$. For $t \in [0, 1]$, let $\gamma(z, t) = \exp_{u(z)} (t \xi^{\rm hor})$ be the family of maps from $\Sigma$ to $P$. Let 
\beqn
\Phi_t: T_{\gamma(z, t)} P \to T_{u(z)} P
\eeqn
be the parallel transport along the curve $\gamma(z, t)$ with respect to the chosen connection $\nabla^{TP}$. Then by the definition of the linearized operator and the properties of the connection (and hence the associated parallel transport)
\beqn
\begin{split}
&\ D_u^{\rm hor}(\xi^{\rm hor})\\
({\rm Definition\ of\ }D_u )\ = &\ {\rm Proj}_{T^{\rm hor} P} \left(  \left. \frac{d}{dt} \right|_{t = 0}   \Phi_t \left( \ov\partial_{J_P} \gamma(z, t) \right) \right)\\
(\Phi_t\ {\rm is\ complex\ linear})\ = &\ \left( {\rm Proj}_{T^{\rm hor}P} \left( \left. \frac{d}{dt} \right|_{ t= 0} \Phi_t ( d_z \gamma(z, t) ) \right) \right)^{0,1}\\
({\rm horizontal}-{\rm vertical\ decomposition})\ = &\ \left( {\rm Proj}_{T^{\rm hor}P} \left( \left. \frac{d}{dt} \right|_{ t= 0} \Phi_t ( d_z^{\rm hor} \gamma(z, t) + d_z^{\rm vert} \gamma(z, t) ) \right) \right)^{0,1}\\
(\Phi_t\ {\rm preserves\ the\ decomposition})\ = &\ \left( \left. \frac{d}{dt} \right|_{ t= 0} \Phi_t ( d_z^{\rm hor} \gamma(z, t) ) \right)^{0,1}\\
({\rm property\ of\ }d_z^{\rm hor})\ = &\ \left( \left. \frac{d}{dt} \right|_{t=0} \Phi_t \left(  d_z (\exp_{\uds u(z)}(t \xi^{\rm hor}))  \right)^{\rm hor} \right)^{0,1}\\
(\nabla^{T^{\rm hor} P} = \pi_P^* \nabla^{TB})\ = &\ D_{\uds u}(\xi^{\rm hor}).
\end{split}
\eeqn
Here in the second last line, $(d_z (\exp_{\uds u(z)}(t \xi^{\rm hor})))^{\rm hor}$ is the horizontal lift. 
\end{proof}

\subsubsection{Sign verification}

Now we turn to the sign verification. Abbreviate ${\mc M} = {\mc M}{}_{0, k+2}^{\rm reg}(J_B, A)$ and $\tilde {\mc M} = \ov{\mc M}{}_{0, k+2}^{\rm reg}(J_P, \tilde A)$. Identify the domain of the pseudocycles $f_i: W_i \to B$ resp. $\tilde f_i: \tilde W_i \to P$ with their images. Suppose 
\beqn
\tilde q \in \ev_P^{-1} \left( \tilde W_0 \times \tilde W_\infty \times \prod_{i=1}^k \tilde W_i \right) \subset \tilde {\mc M}
\eeqn
which projects down to the points
\beqn
q \in \ev_B^{-1} \left( \{p_0 \}\times \{p_\infty\}\times \prod_{i=1}^k W_i \right) \subset {\mc M}.
\eeqn
If $q$ is represented by a smooth $J_B$-holomorphic map $u: S^2 \to B$ with $k$ markings ${\bf z}$, then $\tilde q$ corresponds to an element in the graph moduli 
\beqn
\tilde {\mc M}^{\rm vert}:= \ov{\mc M}{}_{0, 2; {\bf z}}^{\rm graph}(u^* P, J_P|_{u^* P}, \tilde A).
\eeqn

To identify the signs, it suffices to consider the case that all moduli spaces are transverse. In this case, the tangent spaces are kernels of the corresponding linearized operators. Then by Lemma \ref{lemma510} and Lemma \ref{lemma511}, there is a natural exact sequence
\beqn
\xymatrix{ 0 \ar[r] & T_{\tilde q}\tilde {\mc M}^{\rm vert} \ar[r] &  T_{\tilde q}\tilde {\mc M} \ar[r] & T_q {\mc M} 
\ar[r]  & 0}.
\eeqn
As the off-diagonal term in \eqref{upper-triangular} can be turned off in a one-parameter family, one has a canonical orientation-preserving isomorphism
\beqn
\det T_{\tilde q} \tilde {\mc M} \cong \det T_{\tilde q} \tilde {\mc M}^{\rm vert} \otimes \det T_q {\mc M}.
\eeqn
Furthermore, let us consider the differential of the evaluation maps which are isomorphisms onto corresponding normal spaces to the constraints, which are abbreviated using $N$. Then there is a commutative diagram
\beqn
\xymatrix{ 0 \ar[r] & T_{\tilde q}\tilde {\mc M}^{\rm vert} \ar[r] \ar[d]_{D_{\tilde q} \ev^{\rm vert}} &  T_{\tilde q}\tilde {\mc M} \ar[r] \ar[d]_{D_{\tilde q} \ev_P}  & T_q {\mc M}  \ar[d]^{D_q \ev_B}
\ar[r]  & 0 \\
 0 \ar[r] &     \tilde N^{\rm vert} \ar[r] & \tilde N \ar[r] & N  \ar[r] & 0 }.
\eeqn
By taking determinants, one obtains a commutative diagram
\beqn
\vcenter{ \xymatrix{  \det T_{\tilde q} \tilde {\mc M} \ar[rrr]^{\det (D_{\tilde q} \ev_P ) }  \ar[d] &  & &  \det \tilde N \ar[d]  &  \\
 \det T_{\tilde q} \tilde {\mc M}^{\rm vert} \otimes \det T_q {\mc M} \ar[rrr]_{\det(D_{\tilde q} \ev^{\rm vert}) \otimes \det (D_q \ev_B) } & & &  \det \tilde N^{\rm vert} \otimes \det N } }
\eeqn
where vertical arrows are orientation preserving. Now the top horizontal arrow has a sign ${\rm sign}(\tilde q)$ contributing to the GW invariant of $P$, the bottom arrow has a sign ${\rm sign}(\tilde q^{\rm vert})\cdot {\rm sign}(q)$ with the factors contributing to the graph GW invariant and the GW invariant of the base. It follows that 
\beqn
{\rm sign}(\tilde q) =  {\rm sign}(\tilde q^{\rm vert} ) \cdot {\rm sign}(q).
\eeqn

\section{A fibration which does not split}\label{section6}

The cohomological splitting stated in Theorem \ref{thm:split} and Corollary \ref{cor12} potentially fails for certain characteristics, if we drop the the non-divisibility assumption on the Gromov--Witten invariant. Here we provide such an example. 

\begin{thm}\label{thm61}
There exists a smooth Fano variety $B$ (hence rationally connected) and a smooth projective family $M \hookrightarrow P \to B$ such that 
\beqn
H^*(P; {\mb F}_2) \neq H^*(B; {\mb F}_2) \otimes H^*(M; {\mb F}_2).
\eeqn
\end{thm}

We first collect some relevant topological arguments from \cite{Ekedahl_2009}. 

\begin{defn}
Let $G$ be a complex reductive Lie group. A topological space $X_n$ is called an $n$-th cohomological approximation of the classifying space $BG$ if there is a map $X_n \to BG$ inducing an isomorphism
\beqn
H^j(BG; {\mb Z}) \to H^j(X_n; {\mb Z})
\eeqn
for all $j \leq n$.
\end{defn}

Following the argument of \cite[Proposition 2.1]{Ekedahl_2009}, one has the following non-splitting result.

\begin{lemma}
Let $G_{\rm Borel} \subset G$ be a Borel subgroup and $F:= G/G_{\rm Borel}$ the flag variety. Suppose $H^j(BG; {\mb Z})$ has $p$-torsion for some $j$. Then for $n$ sufficiently large and any $n$-th cohomological approximation $X_n$ of $BG$, if $P_n \to X_n$ is the pullback $G$-bundle and $Y_n:= P_n\times_G F$, then 
\beqn
H^*(Y_n; {\mb F}_p) \neq H^*(F; {\mb F}_p ) \otimes H^*(X_n; {\mb F}_p ).
\eeqn
\end{lemma}

\begin{proof}
Consider the universal $F$-fibration $Y_\infty:= EG \times_G F$. When $n$ is sufficiently large, the natural map $Y_n \to Y_\infty$ induces an isomorphism
\beqn
H^j( Y_\infty; {\mb Z}) \cong H^j(Y_n; {\mb Z}).
\eeqn
On the other hand, let $T \subset G_{\rm Borel}$ be the maximal torus. One has
\beqn
Y_\infty = EG\times_G F \cong EG/G_{\rm Borel} \simeq B G_{\rm Borel} \simeq BT
\eeqn
whose cohomology is a polynomial algebra and has no torsion. However, $H^j(X_n; {\mb Z}) \cong H^j(BG; {\mb Z})$ has $p$-torsion. Hence 
\beqn
H^*( Y_n; {\mb Z} ) \neq H^*( F; {\mb Z}) \otimes H^*(X_n; {\mb Z})
\eeqn
as the right hand side has $p$-torsion. The conclusion follows from the universal coefficients theorem.
\end{proof}

\emph{Proof of Theorem 6.1:} Ottem--Rennemo \cite{Ottem_Rennemo_2024} constructed a sequence of cohomological approximations of the classifying spaces for a reductive group $G = GO(4)^\circ$ by a sequence of smooth projective varieties $X_n$ (see \cite[Proposition 3.5]{Ottem_Rennemo_2024}) with algebraic $G$-bundles $P_n \to X_n$. Their computation (\cite[Corollary 3.6]{Ottem_Rennemo_2024}) shows $H^3(BG; {\mb Z}) = {\mb Z}/2$. On the other hand, by \cite[Theorem 4.1]{Ottem_Rennemo_2024}, $X_n$ is Fano. Forming the associated $F=G/G_{\operatorname{Borel}}$ bundle gives a  fibration $Y_n \to X_n$ with fiber the flag variety $F$. To see that that $Y_n$ is projective, choose a $G$-equivariant embedding $F \hookrightarrow \mathbf{P}(\mathbb{V})$, where $\mathbb{V}$ is some linear representation of $G$. The variety $Z_n:=\mathbf{P}(\mathbb{V})\times_G P_n$ is projective, since it is the projectivization of the vector bundle $\mathbb{V}\times_G P_n$ over $X_n$ and the projectivization of any vector bundle over $X_n$ is projective (over $X_n$ and hence over $\mathbb{C}$) \cite[\href{https://stacks.math.columbia.edu/tag/01W7}{Tag 01W7}]{stacks-project}. The variety $Y_n$ is a closed subvariety of $Z_n$ and is thus also projective. Hence, this provides the example required by the statement of Theorem \ref{thm61}.

\begin{rem} \label{rem:rennemo} We sketch an alternative construction of a fibration $P \to B$ satisfying the hypotheses of Theorem \ref{thm61}. The base $B$ will be again be one of the smooth Fano varieties $X_n$ from \cite{Ottem_Rennemo_2024}. There is a non-trivial $\mathbb{P}^1$-bundle over $X_n$ whose Brauer class represents the non-trivial element $b \in H^3(X_n,\mathbb{Z})$ (see \cite[A.8]{Ottem_Rennemo_2024}). Let $h \in H^2(\mathbb{P}^1,\mathbb{Z})$ be the hyperplane class. It follows from the general theory of Brauer classes that on the $E^3$ page of the Leray spectral sequence, we have $$d^3(h)=b.$$ In particular, the spectral sequence fails to degenerate. This construction applies more generally to smooth rationally connected projective varieties with torsion in $H^3(X_n,\mathbb{Z})$.  \end{rem}


\appendix

\section{Pseudocycles in characteristic $p$}\label{app:A}

For semi-positive symplectic manifolds, Gromov--Witten invariants are defined via intersections of pseudocycles (see \cite{Ruan_Tian}\cite{McDuff_Salamon_2004}). One can show that pseudocycles up to cobordism are in one-to-one correspondence with integral homology classes (see \cite{Zinger_2008}). Moreover, as intersection numbers are cobordism invariant, the Gromov--Witten invariants are functions of integral homology classes (modulo torsion).

We would like to consider the analogue of Zinger's theorem in finite characteristic. The results are well-known to experts (cf. \cite{Wilkins_survey}), and we include the discussion for completeness.


\begin{defn}
Let $p$ be a prime number. An oriented $k$-dimensional $p$-pseudocycle in a smooth manifold $X$ is a smooth map $f: W \to X$ from a $k$-dimensional oriented manifold with boundary $W$ to $X$ satisfying the following property:
\begin{enumerate}

\item $f(W)$ is precompact in $X$.

\item There is a $p$-fold oriented covering $\partial W \to V$ and a smooth map $g: V \to X$ such that $f|_{\partial W}$ is the pullback of $g$. 

\item The frontier is small. More precisely, the $\Omega$-set of $f$ is 
\beqn
\Omega_f:= \bigcap_{K \subset W\ {\rm compact}} \ov{ f( W \setminus K)}.
\eeqn
We require that it has dimension at most $k-2$.
\end{enumerate}

\end{defn}

In particular, a pseudocycle is a $p$-pseudocycle with $\partial W = \emptyset$.

Let $\tilde {\mc H}_k^{(p)} (X)$ be the set of all oriented $k$-dimensional $p$-pseudocycles. It has the structure of an abelian group where the sum of two $p$-pseudocycles is their disjoint union and the inverse of a $p$-pseudocycle is the same map with domain orientation reversed.

\begin{defn}
Two oriented $k$-dimensional $p$-pseudocycles $f_0: W_0 \to X$, $f_1: W_1 \to X$ are said to be $p$-cobordant if there exists another smooth map $g: V \to X$ from an oriented $k$-dimensional manifold $V$ without boundary, a smooth map $\tilde f: \tilde W \to X$ from an oriented $k+1$-dimensional manifold $\tilde W$ with boundary with 
\beqn
\tilde f(\tilde W)\ {\rm is\ precompact},\ {\rm dim} \Omega_{\tilde f} \leq k-1,
\eeqn
an oriented diffeomorphism
\beqn
\partial \tilde W  \cong - {\rm Int} W_0 \sqcup {\rm Int} W_1 \sqcup W'
\eeqn
where $W'$ is a $k$-dimensional manifold without boundary (having the induced orientation)
such that the restriction of $\tilde f$ to ${\rm Int} W_0$ resp. ${\rm Int} W_1$ coincides with $f_0$ resp. $f_1$, and an oriented $p$-fold covering $W' \to V$ such that $\tilde f|_{W'}$ is the pullback of $g$. 
\end{defn}

One needs to do some simple modifications to make certain naive operations satisfy the above definition. For example, a $p$-pseudocycle $f: W \to X$ is $p$-cobordant to itself. However, the naive map 
\beqn
\tilde f: W \times [0, 1] \to X,\ \tilde f(x, t) = f(x)
\eeqn
is a $p$-cobordism only after removing the corner $\partial W \times \{0, 1\}$.

It is clear that $p$-cobordant is an equivalence relation and respect the additive structure. Let ${\mc H}_k^{(p)} (X)$ be the set of $p$-cobordant classes of $p$-pseudocycles, which is an abelian group. As $p$ times of any $p$-pseudocycle is $p$-cobordant to the empty set via the empty $p$-cobordism, ${\mc H}_k^{(p)}(X)$ is indeed an ${\mb F}_p$-vector space.

Before we discuss the relation between pseudocycles and homology classes, we define the intersection pairing between pseudocycles in finite characteristic. 

Let 
\beqn
f_1: W_1 \to X,\ f_2: W_2 \to X
\eeqn
be two oriented $p$-pseudocycles in $X$ of complimentary dimensions. When they intersect transversely, meaning that $\ov{f_1(W_1)} \cap \ov{f_2(W_2)} = f_1({\rm Int}W_1) \cap f_2({\rm Int} W_2)$ and the intersection is transverse, the signed count of intersection points, modulo $p$, is defined to be the intersection number. One can see easily that this intersection number is invariant under $p$-cobordism, provided that the $p$-cobordism is in general position. 

In addition, if $f_1: W_1 \to X$ is $p$-cobordant to the empty set via a $p$-cobordism $\tilde f_1: \tilde W_1 \to X$, and if $\tilde f_1$ and $f_2$ are transverse, then $\tilde f_1(\tilde W_1) \cap f_2(W_2)$ is a compact oriented 1-dimensional manifold with boundary being
\beqn
\Big( \tilde f_1( \partial \tilde W_1) \cap f_2({\rm Int} W_2)\Big) \sqcup \Big( \tilde f_1( {\rm Int} \tilde W_1) \cap f_2(\partial W_2) \Big).
\eeqn
Besides the intersection $f_1({\rm Int} W_1) \cap f_2({\rm Int} W_2)$, other boundary intersections contribute to a multiple of $p$. 

In general, if relevant intersections are not transverse, then one can perturb via ambient diffeomorphisms to achieve transversality. This allows us to define intersection numbers between any pair of $p$-pseudocycles of complementary dimensions and prove the independence of the choice of perturbations, as all nearby diffeomorphisms are homotopic. Therefore, we have defined a bilinear pairing 
\beqn
{\mc H}^{(p)}(X) \otimes {\mc H}^{(p)}(X) \to {\mb F}_p
\eeqn

Next we will define the map from homology to pseudocycles. Recall that one also has an abelian group ${\mc H}_k(X)$ of genuine $k$-dimensinoal pseudocycles up to genuine cobordism. There is an obvious group homomorphism
\beqn
{\mc H}_* (X) \to {\mc H}_* ^{(p)} (X).
\eeqn
Zinger \cite{Zinger_2008} constructed a natural isomorphism
\beqn
\Phi_*: H_* (X; {\mb Z}) \cong {\mc H}_* (X).
\eeqn

\begin{thm}\label{pseudocycle_theorem}
There is a natural map 
\beqn
\Phi^{(p)}_*: H_* (X; {\mb F}_p) \to {\mc H}_*^{(p)}(X)
\eeqn
satisfying the following conditions.
\begin{enumerate}
\item The following diagram commutes.
\beq\label{pseudocycle_commutative_diagram}
\vcenter{ \xymatrix{  H_* (X; {\mb Z}) \ar[r]^{\Phi_*} \ar[d] & {\mc H}_* (X) \ar[d] \\
            H_* (X; {\mb F}_p) \ar[r]_{\Phi_*^{(p)}} &    {\mc H}_*^{(p)}(X)  } }
\eeq

\item Suppose a homology class $a \in H_*(X; {\mb F}_p)$ is represented by a smooth cycle $f: W \to X$ where $W$ is a compact oriented manifold with boundary such that $f|_{\partial W}$ is a $p$-fold oriented covering of a map $g: V \to X$ from a compact oriented manifold $V$ without boundary; in particular, $f$ is a $p$-pseudocycle. Then $\Phi_*^{(p)}(a)$ is represented by $f$.

\item The map $\Phi_*^{(p)}$ intertwines the Poincar\'e pairing on $H_* (X; {\mb F}_p)$ with the intersection pairing on ${\mc H}_*^{(p)}(X)$.
\end{enumerate}
\end{thm}
\begin{proof}
For the purpose of this paper, we do not need to show that $\Phi_*^{(p)}$ is invertible. We just need to consider the case that $p>2$ as it is known that for $p=2$ any homology class can be represented by a closed submanifold (see, e.g., \cite[Theorem B]{thom-elementary})

We follow the same approach as Zinger \cite{Zinger_2008}. Consider the complex of singular chains with ${\mb F}_p$-coefficients. A homology class in ${\mb F}_p$ coefficients is represented by a singular cycle
\beqn
C = \sum_{i=1}^N a_i h_i
\eeqn
where $a_i \in {\mb F}_p$ and $h_i: \Delta_k \to X$ is a continuous map from the $k$-simplex $\Delta_k$. We can always choose the representative such that each $h_i$ is smooth. Let $\delta^j(h_i)$ be the $j$-th face of $h_i$, which is a $k-1$-simplex. Moreover, let $\tilde a_i \in \{1, \ldots, p-1\}$ be a lift of $a_i$ in ${\mb Z}$. For any smooth map $g: \Delta_{k-1} \to X$ which may appear as a face of $h_i$, consider 
\beqn
\Delta_{i,g}^\pm(C) = \Big\{ j\ |\delta^j(h_i) = \pm g \Big\}
\eeqn
and 
\beqn
\Delta_g^\pm (C) = \bigcup_{i=1}^N \left( \underbrace{\Delta_{i, g}^\pm(C) \sqcup \cdots \sqcup \Delta_{i, g}^\pm (C) }_{\tilde a_i} \right)
\eeqn
Then since $C$ is a cycle, one has 
\beqn
\# \Delta_g^+(C) - \# \Delta_g^-(C)  \in p{\mb Z}.
\eeqn
Without loss of generality, assume $\# \Delta_g^+(C) \geq \# \Delta_g^-(C)$. Then choose an injection $\Delta_g^-(C) \hookrightarrow \Delta_g^+(C)$ and glue the corresponding (interior of) faces, and removing all codimension two or higher facets from the simplexes, one obtains a topological $k$-manifold with boundary $N(C)$ together with a map $f: N(C) \to X$. Notice that the boundary of $N(C)$ can be identified as $p$ copies of a manifold. If we fix a certain standard way of gluing standard $k$-simplexes along a face, $N(C)$ is then equipped with a canonical smooth structure. One can perturb $f: N(C) \to X$ to a smooth map, and hence a smooth $p$-pseudocycle. The cobordism class of the pseudocycle $f$ is independent of the choice of the perturbation.

On the other hand, by using the same method as Zinger \cite[Lemma 3.3]{Zinger_2008}, one can show that two homologous cycles induce cobordant $p$-pseudocycles. The details are left to the reader. Hence we have constructed the map 
\beqn
\Phi_*^{(p)}: H_*(X; {\mb F}_p) \to {\mc H}_*^{(p)} (X).
\eeqn
It is easy to see that this is an ${\mb F}_p$-linear map making the diagram \eqref{pseudocycle_commutative_diagram} commutative and satisfying (2) of Theorem \ref{pseudocycle_theorem}. Moreover, by comparing the definition of the Poincar\'e pairing (which is essentially counting transverse intersections of cycles) and the intersection pairing, (3) of Theorem \ref{pseudocycle_theorem} is also true. 
\end{proof}

\bibliography{reference}

\providecommand{\bysame}{\leavevmode\hbox to3em{\hrulefill}\thinspace}
\providecommand{\MR}{\relax\ifhmode\unskip\space\fi MR }
\providecommand{\MRhref}[2]{%
  \href{http://www.ams.org/mathscinet-getitem?mr=#1}{#2}
}
\providecommand{\href}[2]{#2}
\begin{thebibliography}{KMM92}

\bibitem[AB24]{Abouzaid_Blumberg_2024}
Mohammed Abouzaid and Andrew Blumberg, \emph{Foundation of {F}loer homotopy theory {I}: flow categories}, \url{http://arxiv.org/abs/2404.03193}, 2024.

\bibitem[AMS21]{AMS}
Mohammed Abouzaid, Mark McLean, and Ivan Smith, \emph{Complex cobordism, {H}amiltonian loops and global {K}uranishi charts}, \url{http://arxiv.org/abs/2110.14320}, 2021.

\bibitem[AMS23]{AMS2}
\bysame, \emph{Gromov--{W}itten invariants in complex-oriented cohomology theories}, \url{https://arxiv.org/abs/2307.01883}, 2023.

\bibitem[Baa70]{Baas}
Nils~Andreas Baas, \emph{Bordism theories with singularities}, Proceedings of the {A}dvanced {S}tudy {I}nstitute on {A}lgebraic {T}opology ({A}arhus {U}niv., {A}arhus, 1970), {V}ol. {I}, Various Publications Series, vol. No. 13, Aarhus Univ., Aarhus, 1970, pp.~1--16.

\bibitem[BGS88]{BGS2}
Jean-Michel Bismut, Henri Gillet, and Christophe Soul\'e, \emph{Analytic torsion and holomorphic determinant bundles. {II}. direct images and {B}ott--{C}hern forms}, Communications in Mathematical Physics \textbf{115} (1988), 79--126.

\bibitem[BH81]{thom-elementary}
Sandro Buoncristiano and Derek Hacon, \emph{An elementary geometric proof of two theorems of {Thom}}, Topology \textbf{20} (1981), 97--99.

\bibitem[Bla56]{Blanchard}
Andr\'e Blanchard, \emph{Sur les vari\'et\'es analytiques complexes}, Annales scientifiques de ENS (1956), no.~2, 157--202.

\bibitem[BP24]{Bai_Pomerleano_2024}
Shaoyun Bai and Daniel Pomerleano, \emph{Equivariant formality in complex-oriented theories}, \url{https://arxiv.org/abs/2405.05821}, 2024.

\bibitem[BX22a]{Bai_Xu_Arnold}
Shaoyun Bai and Guangbo Xu, \emph{Arnold conjecture over integers}, \url{http://arxiv.org/abs/2209.08599}, 2022.

\bibitem[BX22b]{Bai_Xu_2022}
\bysame, \emph{An integral {E}uler cycle in normally complex orbifolds and {$\mathbb Z$}-valued {G}romov--{W}itten type invariants}, \url{https://arxiv.org/abs/2201.02688}, 2022.

\bibitem[BX24]{Bai_Xu_Floer}
\bysame, \emph{Integral {H}amiltonian {F}loer theory}, Upcoming, 2024.

\bibitem[Cam92]{campana}
Fr\'ed\'eric Campana, \emph{Connexit\'{e} rationnelle des vari\'{e}t\'{e}s de {F}ano}, Ann. Sci. \'{E}cole Norm. Sup. (4) \textbf{25} (1992), no.~5, 539--545.

\bibitem[Del68]{Deligne}
P.~Deligne, \emph{Th\'{e}or\`eme de {L}efschetz et crit\`eres de d\'{e}g\'{e}n\'{e}rescence de suites spectrales}, Inst. Hautes \'{E}tudes Sci. Publ. Math. (1968), no.~35, 259--278. \MR{244265}

\bibitem[Eke09]{Ekedahl_2009}
Torsten Ekedahl, \emph{Approximating classifying spaces by smooth projective varieties}, \url{http://arxiv.org/abs/0905.1538}, 2009.

\bibitem[FO97]{Fukaya_Ono_integer}
Kenji Fukaya and Kaoru Ono, \emph{Floer homology and {G}romov--{W}itten invariant over integer of general symplectic manifolds - summary -}, Proceeding of the {L}ast {T}aniguchi Conference, 1997.

\bibitem[GHS03]{ghs03}
Tom Graber, Joe Harris, and Jason Starr, \emph{Families of rationally connected varieties}, J. Amer. Math. Soc. \textbf{16} (2003), no.~1, 57--67.

\bibitem[Har01]{Harris2001}
Joe Harris, \emph{Lectures on rationally connected varieties}, avalaible at: https://mat.uab.cat/~kock/RLN/rcv.pdf, 2001.

\bibitem[Hir64]{Hironaka_resolution}
Heisuke Hironaka, \emph{Resolution of singularities of an algebraic variety over a field of characteristic zero: {I}}, Annals of Mathematics \textbf{71} (1964), no.~1, 109--203.

\bibitem[HS22]{Hirschi_Swaminathan}
Amanda Hirschi and Mohan Swaminathan, \emph{Global {K}uranishi charts and a product formula in symplectic {G}romov--{W}itten theory}, \url{https://arxiv.org/abs/2212.11797}, 2022.

\bibitem[Hu19]{hu-rational-connected}
Jianxun Hu, \emph{Gromov-{W}itten invariants and rationally connectedness}, Gromov-{W}itten theory, gauge theory and dualities, Proc. Centre Math. Appl. Austral. Nat. Univ., vol.~48, Austral. Nat. Univ., Canberra, 2019, p.~23.

\bibitem[Hyv12]{hyvrier}
Cl\'{e}ment Hyvrier, \emph{A product formula for {G}romov-{W}itten invariants}, J. Symplectic Geom. \textbf{10} (2012), no.~2, 247--324.

\bibitem[KMM92]{Kollar_Miyaoka_Mori_1992}
J{\'a}nos Koll{\'a}r, Yoichi Miyaoka, and Shigefumi Mori, \emph{Rational connectedness and boundedness of {F}ano manifolds}, Journal of Differential Geometry \textbf{36} (1992), no.~3, 765--779.

\bibitem[Kol98]{kollar-ems}
J\'{a}nos Koll\'{a}r, \emph{Low degree polynomial equations: arithmetic, geometry and topology}, European {C}ongress of {M}athematics, {V}ol. {I} ({B}udapest, 1996), Progr. Math., vol. 168, Birkh\"{a}user, Basel, 1998, pp.~255--288.

\bibitem[Kol01]{kollar2001}
\bysame, \emph{Which are the simplest algebraic varieties?}, Bull. Amer. Math. Soc. (N.S.) \textbf{38} (2001), no.~4, 409--433. \MR{1848255}

\bibitem[Las79]{Lashof_1979}
Richard Lashof, \emph{Stable ${G}$-smoothing}, Algebraic Topology Waterloo 1978, Springer, 1979, pp.~283--306.

\bibitem[LM03]{Lalonde_McDuff_2003}
Fran\c{c}ois Lalonde and Dusa McDuff, \emph{Symplectic structures on fiber bundles}, Topology \textbf{42} (2003), no.~2, 309--347.

\bibitem[LMP99]{LMP}
Fran\c{c}ois Lalonde, Dusa McDuff, and Leonid Polterovich, \emph{Topological rigidity of {H}amiltonian loops and quantum homology}, Inventiones Mathematicae \textbf{135} (1999), no.~2, 369--385.

\bibitem[McD00]{mcduffseidel}
Dusa McDuff, \emph{Quantum homology of fibrations over {$S^2$}}, Internat. J. Math. \textbf{11} (2000), no.~5, 665--721.

\bibitem[Mil64]{Milnor_micro_1}
John Milnor, \emph{Microbundles {Part I}}, Topology \textbf{3} (1964), no.~Suppl. 1, 53--80.

\bibitem[MS98]{McDuff_Salamon_1998}
Dusa McDuff and Dietmar Salamon, \emph{Introduction to symplectic topology}, 2nd ed., Oxford University Press, New York, 1998.

\bibitem[MS04]{McDuff_Salamon_2004}
\bysame, \emph{${J}$-holomorphic curves and symplectic topology}, Colloquium Publications, vol.~52, American Mathematical Society, 2004.

\bibitem[OR24]{Ottem_Rennemo_2024}
John Ottem and J{\o}rgen Rennemo, \emph{Fano varieties with torsion in the third cohomology group (appendix by j{\'a}nos) koll{\'a}r}, Journal f{\"u}r die reine und angewandte Mathematik (Crelles Journal) (2024), no.~0.

\bibitem[Par13]{BParker_integer}
Brett Parker, \emph{Integral counts of pseudo-holomorphic curves}, arXiv: 1309.0585, 2013.

\bibitem[Par21]{pardon2020orbifold}
John Pardon, \emph{Orbifold bordism and duality for finite orbispectra}, Geom. Topol. (to appear) (2021), 1--74.

\bibitem[RT95]{Ruan_Tian}
Yongbin Ruan and Gang Tian, \emph{A mathematical theory of quantum cohomology}, Journal of Differential Geometry \textbf{42} (1995), 259--367.

\bibitem[Sch99]{Schwarz_equivalence}
Matthias Schwarz, \emph{Equivalences for {M}orse homology}, Contemporary Mathematics (1999).

\bibitem[Sei97]{seidel-rep}
P.~Seidel, \emph{{$\pi_1$} of symplectic automorphism groups and invertibles in quantum homology rings}, Geom. Funct. Anal. \textbf{7} (1997), no.~6, 1046--1095.

\bibitem[{Sta}24]{stacks-project}
The {Stacks project authors}, \emph{The stacks project}, \url{https://stacks.math.columbia.edu}, 2024.

\bibitem[Tia12]{tian-threefold}
Zhiyu Tian, \emph{Symplectic geometry of rationally connected threefolds}, Duke Math. J. \textbf{161} (2012), no.~5, 803--843.

\bibitem[Tia15]{tian-fourfold}
\bysame, \emph{Symplectic geometry and rationally connected 4-folds}, J. Reine Angew. Math. \textbf{698} (2015), 221--244.

\bibitem[Voi08]{voisin-threefold}
Claire Voisin, \emph{Rationally connected 3-folds and symplectic geometry}, Ast\'{e}risque (2008), no.~322, 1--21, G\'{e}om\'{e}trie diff\'{e}rentielle, physique math\'{e}matique, math\'{e}matiques et soci\'{e}t\'{e}. II.

\bibitem[Weh12]{Wehrheim-Morse}
Katrin Wehrheim, \emph{Smooth structures on {M}orse trajectory spaces, featuring finite ends and associative gluing}, Proceedings of the {F}reedman {F}est, Geometry {\&} Topology Monographs, vol.~18, Geom. Topol. Publ., Coventry, 2012, pp.~369--450.

\bibitem[Wil24]{Wilkins_survey}
Nicholas Wilkins, \emph{A survey of equivariant operations on quantum cohomology for symplectic manifolds}, \url{https://sites.google.com/site/nwilkinsmaths/a-survey-of-equivariant-quantum-operations?authuser=0}, 2024.

\bibitem[Zin08]{Zinger_2008}
Aleksey Zinger, \emph{Pseudocycles and integral homology}, Transactions of the American Mathematical Society \textbf{360} (2008), 2741--2765.

\end{thebibliography}

\bibliographystyle{amsalpha}

\end{document}